\documentclass[12pt,letterpaper]{amsart}
\usepackage{fullpage}
\usepackage{hyperref}
\usepackage{amssymb}
\usepackage{enumerate}
\usepackage{amsmath}
\usepackage{amsthm}
\usepackage{amssymb}
\usepackage{mathtools}
\usepackage[usenames,dvipsnames]{color}
\usepackage{float}
\usepackage{tikz}
\usepackage{tikz-cd}
\usepackage[utf8]{inputenc}
\usepackage[english]{babel}
\usepackage{amsmath,calligra,mathrsfs}
\usepackage[capitalize,noabbrev]{cleveref}

\usepackage[backend=bibtex,sorting=anyt]{biblatex}
\bibliography{References}

\usetikzlibrary{matrix}

\newtheorem{theorem}{Theorem}[section]

\newtheorem{corollary}[theorem]{Corollary}
\newtheorem{proposition}[theorem]{Proposition}

\theoremstyle{definition}
\newtheorem{definition}[theorem]{Definition}

\theoremstyle{remark}
\newtheorem{remark}[theorem]{Remark}
\theoremstyle{lemma}
\newtheorem{lemma}[theorem]{Lemma}
\theoremstyle{example}
\newtheorem{example}[theorem]{Example}

\newcommand{\Hom}{\text{Hom}}

\DeclareMathOperator{\colim}{colim}

\newcommand{\indsimp}[1]{|\Delta^{#1}|^{J_1}}
\newcommand{\disimp}[1]{|\Delta^{#1}|^{J_+}}
\newcommand{\gsimp}[1]{|\Delta^{#1}|^{J}}

\newcommand{\itopcube}[1]{|\Box^{#1}|^{(I,\boxdot)}}
\newcommand{\indcube}[1]{|\Box^{#1}|^{(J_1,\times)}}
\newcommand{\iindcube}[1]{|\Box^{#1}|^{(J_1,\boxdot)}}
\newcommand{\dicube}[1]{|\Box^{#1}|^{(J_+,\times)}}
\newcommand{\idicube}[1]{|\Box^{#1}|^{(J_+,\boxdot)}}
\newcommand{\gcube}[1]{|\Box^{#1}|^{(J,\otimes)}}

\newcommand{\topchain}[2]{C^{I}_{#1}(#2)}
\newcommand{\indsimpchain}[2]{C^{J_1}_{#1}(#2)}
\newcommand{\disimpchain}[2]{C^{J_+}_{#1}(#2)}
\newcommand{\gsimpchain}[2]{C^{J}_{#1}(#2)}
\newcommand{\gsimpchainnorm}[2]{\overline{C}^{J}_{#1}(#2)}
\newcommand{\gsimpchainred}[2]{\tilde{C}^{J}_{#1}(#2)}
\newcommand{\gsimpcochain}[2]{C_{J}^{#1}(#2)}

\newcommand{\topcubechain}[2]{C_{#1}^{(I,\times)}(#2)}

\newcommand{\indcubechain}[2]{C^{(J_1,\times)}_{#1}(#2)}
\newcommand{\iindcubechain}[2]{C^{(J_1,\boxdot)}_{#1}(#2)}
\newcommand{\dicubechain}[2]{C^{(J_+,\times)}_{#1}(#2)}
\newcommand{\idicubechain}[2]{C^{(J_+,\boxdot)}_{#1}(#2)}
\newcommand{\gcubechain}[2]{C^{(J,\otimes)}_{#1}(#2)}

\newcommand{\gcubesimpchain}[2]{C^{(J,\times)}_{#1}(#2)}
\newcommand{\gcubesimpchainred}[2]{\tilde{C}^{(J,\times)}_{#1}(#2)}
\newcommand{\gcubechaindeg}[2]{D^{(J,\otimes)}_{#1}(#2)}

\newcommand{\gcubechainwhole}[2]{Q^{(J,\otimes)}_{#1}(#2)}

\newcommand{\gcubecochain}[2]{C_{(J,\otimes)}^{#1}(#2)}
\newcommand{\gcubesimpcochain}[2]{C_{(J,\times)}^{#1}(#2)}

\newcommand{\tophom}[2]{H^{I}_{#1}(#2)}
\newcommand{\indhom}[2]{H^{J_1}_{#1}(#2)}
\newcommand{\dihom}[2]{H^{J_+}_{#1}(#2)}
\newcommand{\gsimphom}[2]{H^{J}_{#1}(#2)}

\newcommand{\gsimphomred}[2]{\tilde{H}^{J}_{#1}(#2)}
\newcommand{\gsimpcohom}[2]{H_{J}^{#1}(#2)}

\newcommand{\gcubesimphom}[2]{H^{(J,\times)}_{#1}(#2)}
\newcommand{\gcubehom}[2]{H^{(J,\otimes)}_{#1}(#2)}

\newcommand{\gcubehomred}[2]{\tilde{H}^{(J,\otimes)}_{#1}(#2)}
\newcommand{\gcubecohom}[2]{H_{(J,\otimes)}^{#1}(#2)}
\newcommand{\gicubecohom}[2]{H_{(J,\boxdot)}^{#1}(#2)}
\newcommand{\gcubesimpcohom}[2]{H_{(J,\times)}^{#1}(#2)}

\definecolor{OliveGreen}{rgb}{0.33, 0.42, 0.18}

\newcommand{\cat}{\mathbf}

\newcommand{\eps}{\varepsilon}

\title{Eilenberg-Steenrod homology and cohomology theories for \v{C}ech's closure spaces}
\author{Peter Bubenik, Nikola Mili\'cevi\'c}

\tikzset{%
    symbol/.style={%
        draw=none,
        every to/.append style={%
            edge node={node [sloped, allow upside down, auto=false]{$#1$}}}
    }
}

\begin{document}

\maketitle
\begin{abstract}
    We generalize some of the fundamental results of algebraic topology from topological spaces to \v{C}ech's closure spaces, also known as pretopological spaces. Using simplicial sets and cubical sets with connections, we define three distinct singular (relative) simplicial and six distinct singular (relative) cubical (co)homology groups of closure spaces. Using acyclic models we show that the three simplicial groups have isomorphic cubical analogues among the six cubical groups. Thus, we obtain a total of six distinct singular (co)homology groups of closure spaces. Each of these is shown to have a compatible homotopy theory that depends on the choice of a product operation and an interval object. We give axioms for an Eilenberg-Steenrod (co)homology theory with respect to a product operation and an interval object. We verify these axioms for three of our six (co)homology groups. For the other three, we verify all of the axioms except excision, which remains open for future work. We also show the existence of long exact sequences of (co)homology groups coming from K\"unneth theorems and short exact (co)chain complex sequences of pairs of closure spaces.  %Finally, we briefly discuss (co)sheaves on closure spaces.
\end{abstract}

\section{Introduction}
\v{C}ech closure spaces, also known as pretopological spaces, are a generalization of topological spaces. There is an axiomatization of topological spaces based on the notion of a closure operator instead of the notion of an open set. In this axiomatization, if one omits the axiom that the closure operator is idempotent, one obtains \v{C}ech closure spaces, which we will refer to simply as closure spaces. We are interested in this generalization because it contains mathematical structures of interest in both pure and applied topology, such as simple directed and undirected graphs, digital images, and topological spaces. Furthermore,  metric spaces and weighted graphs may be thought of as one-parameter families of closure spaces. From closure spaces there are functorial constructions of \v{C}ech and Vietoris-Rips simplicial complexes and their homology groups \cite{bubenik2021applied,rieser2020vietoris,vela2019homology}. Closure spaces are thus a natural setting for extending classical results from algebraic topology for use in a number of scientific, engineering and data analysis applications.
They have been used in image analysis \cite{lamure1987espaces,bonnevay2009pretopological}, shape recognition 
\cite{emptoz1983modele,frelicot1998pretopological}, supervised learning \cite{frank1996pretopological,frelicot1998pretopology} 
and complex systems modeling 
\cite{ahat2009pollution,largeron2002pretopological} and recently have been considered as an alternative in developing the theory of applied topology \cite{rieser2017cech,bubenik2021applied}.
The framework of \v{C}ech closure spaces was shown to be somewhat universal, as seemingly one can recover a huge body of work spanning the last several decades in applied topology, by translating the work in the language of closure spaces. Persistent homology theory along with multiple stability results was also developed for closure spaces. Several different notions of homotopy and homology groups in seemingly different applications can also be understood from the point of view of closure spaces \cite{bubenik2021applied}. This work is a continuation of the work in \cite{bubenik2021applied}. There the authors introduced several homology theories for closure spaces and left open questions regarding the properties these homology groups have. For example, are there analogues of the Eilenberg-Steenrod axioms for those homology groups, Mayer Vietories long exact sequences, excision or K\"unneth theorems. This is what this work attempts to do; to lay foundations for the theory of algebraic pretopology. 

We use simplicial sets \cite{goerss2009simplicial} and cubical sets with connections \cite{tonks1992cubical} together with Dold-Kan correspondences \cite{goerss2009simplicial,brown2004nonabelian} to define multiple homology theories for closure spaces, whose definitions coincide with the homology groups introduced in \cite{bubenik2021applied}. Our building blocks are three closure spaces of interest; $I$ (the unit interval) and $J_+$ and $J_1$ which are combinatorial objects ($J_+$ is a simple directed graph and $J_1$ is a simple undirected graph). Given $J\in \{J_+,J_1,I\}$ and $\otimes\in \{\boxdot, \times\}$ where $\boxdot$ and $\times$ are two canonical product operations on the category of closure spaces \cite{bubenik2021applied,cech1966topological}, we define singular simplicial chain complexes $\gsimpchain{\bullet}{X}$ and their singular simplicial homology groups $\gsimphom{\bullet}{X}$ and singular cubical chain complexes $\gcubechain{\bullet}{X}$ and their singular cubical homology groups $\gcubehom{\bullet}{X}$. The choice $J\in \{J_+,J_1,I\}$ and $\otimes\in \{\boxdot,\times\}$ also gives rise to an equivalence relation on the set of continuous maps $\cat{Cl}(X,Y)$ between closure spaces $X$ and $Y$, $f\sim_{(J,\otimes)}$ which we think of as a homotopy relation \cite{bubenik2021applied}. A few of the major results of this paper are now listed. First, we have the following theorem concerning the homotopy invariance of our homology groups.

\begin{theorem}[Homotopy invariance]
Let $f,g:X\to Y$ be two continuous maps of closure spaces and suppose that  $f\sim_{(J,\otimes)}g$. Then $f_*=g_*:\gcubehom{\bullet}{X}\to \gcubehom{\bullet}{Y}$ (\cite[Theorem 6.6]{bubenik2021applied}). Suppose that $f\sim_{(J,\times)}g$. Then $f_*=g_*:\gsimphom{n}{X}\to \gsimphom{n}{Y}$ for all $n\ge 0$. 
\end{theorem}

Using the acyclic models theory by Eilenberg and MacLane \cite{eilenberg1953acyclic}, we also show that there are chain homotopies between the singular simplicial chain complexes and some singular cubical chain complexes we define.

\begin{theorem}
For each closure space $X$ there exist chain homotopies $\gsimpchain{\bullet}{X}\simeq \gcubesimpchain{\bullet}{X}$.
\end{theorem}

Furthermore, the acyclic models theorem also allows us to prove the Eilenberg-Zilber theorem for closure spaces, which in turn leads to K\"unneth theorems.

\begin{theorem}[Eilenberg-Zilber]
Let $X$ and $Y$ be closure spaces. Then there are natural chain homotopy equivalences $\gsimpchain{\bullet}{X\times Y}\simeq \gsimpchain{\bullet}{X}\otimes \gsimpchain{\bullet}{Y}$ and $\gcubesimpchain{\bullet}{X\times Y}\simeq \gcubesimpchain{\bullet}{X}\otimes \gcubesimpchain{\bullet}{Y}$.
\end{theorem}

Other analogues of classical results from algebraic topology such as Mayer-Vietoris long exact sequences and excision theorems are also presented in this work.

\begin{theorem}[Mayer-Vietoris]
Let $X$ be a closure space and let $\{A,B\}$ be an interior cover of $X$. Then we have the following long exact sequences:
\begin{gather*}\cdots\to \gsimphom{n}{A\cap B}\to \gsimphom{n}{A}\oplus \gsimphom{n}{B}\to \gsimphom{n}{X}\to \gsimphom{n-1}{A}\to\cdots\\
\cdots\to \gcubesimphom{n}{A\cap B}\to \gcubesimphom{n}{A}\oplus \gcubesimphom{n}{B}\to \gcubesimphom{n}{X}\to \gcubesimphom{n-1}{A}\to\cdots
\end{gather*}
\end{theorem}

\begin{theorem}[Excision]
Let $(X,c)$ be a closure space and $i$ the induced interior operator. Let $Z\subseteq A\subseteq X$ be such that $c(Z)\subseteq i(A)$. Then, the inclusion $(X-Z,A-Z)\to (X,A)$ induces isomorphisms $\gsimphom{n}{X-Z,X-A)}\cong \gsimphom{n}{X,A}$ and $\gcubesimphom{n}{X-Z,X-A}\cong \gcubesimphom{n}{X,A}$ for all $n$. Equivalently, for an interior cover $\{A,B\}$ of $(X,c)$, the inclusion $(B,A\cap B)\to (X,A)$ induces isomorphisms $\gsimphom{n}{B,A\cap B}\cong \gsimphom{n}{X,A}$ and $\gcubesimphom{n}{B,A\cap B}\cong \gcubesimphom{n}{X,A}$ for all $n$.
\end{theorem}

 We also use these results to formulate the analogues of Eilenberg-Steenrod axioms for a homology theory of closure spaces and then show existence for three of these.  %We also reintroduce sheaf theory for closure spaces from \cite{vcech1966topological}, using modern language. We give an interesting application to simple directed graphs. 
 From the simplicial and cubical singular chain complexes we define we can also dually define simplicial and cubical singular cochain complexes and their respective cohomology groups. By duality, all of the above results for homology groups translate to cohomology statements. Besides these major results, the paper is full of examples and remarks about proof strategies and ideas that can be used for future work in the field.

This paper is structured as follows. In \cref{section:background} we provide background and results for closure spaces. In \cref{section:closure_spaces_and_simplicial_sets} we review simplicial sets and cubical sets with connections. We then use Dold-Kan correspondences to define homology groups of closure spaces. In \cref{section:homotopy} we recall notions of homotopy for closure spaces and show that our homology groups have corresponding compatible homotopy invariance properties. In \cref{section:homology} we prove analogues of classical results from algebraic topology for our homology groups such as Eilenberg-Steenrod axioms and Mayer-Vietories long exact sequences and K\"unneth theorems. In \cref{section:cohomology} we define multiple cohomology groups of closure spaces and also show analogues of classical results for these. %In \cref{section:sheaves} we briefly discuss sheaves and cosheaves of abelian groups on closure spaces.

\subsection*{Related work}
Extending homotopy theory from topological spaces to closure spaces has been considered by Rieser \cite{rieser2017cech} and Demaria and Bogin \cite{demaria1984homotopy,demaria1985shape}. The homotopy notions and homology groups studied in this paper were originally  developed in \cite{bubenik2021applied} where the homotopy invariance of some of the homology groups was proven. \v{C}ech (co)homology for closure spaces was developed in the PhD thesis of Palacios \cite{vela2019homology}. Palacios also showed Mayer-Vietoris long exact sequences and Eilenberg-Steenord axiomatization in their work.
A collection of closure spaces of interest is given by symmetric and reflexive relations, or equivalently, by simple undirected graphs. These were studied by Poincar\'e, who called them physical continua \cite{poincare2012value}, by Zeeman, who called them tolerance spaces 
\cite{zeeman1962topology}, and by Poston, who called them fuzzy spaces \cite{poston1971fuzzy}. In the language of closure spaces, they are also known as semi-uniform quasi-discrete spaces \cite{cech1966topological,bubenik2021applied}. Vietoris-Rips homology groups for semi-uniform spaces were introduced by Rieser in \cite{rieser2020vietoris} and analogues of Eilenberg-Steenrod axioms were shown for these. When the underlying set is a finite subset of the integer lattice $\mathbb{Z}^n$,
they are called digital images and are studied in digital topology \cite{kong1992concepts}. $A$-theory, also called discrete homotopy theory, was was developed for simple undirected graphs \cite{kramer1998combinatorial,barcelo2001foundations,barcelo2005perspectives,babson2006homotopy}.
Barcelo, Capraro and White \cite{barcelo2014discrete} extended this work and developed a compatible discrete homology theory for metric spaces. Our work extends these theories to  closure spaces. 
A directed analogue of $A$-homotopy for simple directed graphs has also been studied \cite{grigor2014homotopy}. More recently, Dochtermann and Singh, defined several homotopy theories for directed graphs, using both cylinder and path objects \cite{dochtermann2021homomorphism}.
 Our work also extends these ideas to closure spaces. In \cite{rieser2021grothendieck}, Rieser introduces a novel sheaf theory for closure spaces. %We show our (co)homology groups are examples of such (co)sheaves. 

\section{Preliminaries}
\label{section:background}
In this section we provide background on \v{C}ech closure spaces. For more details, see  \cite{vcech1966topological,rieser2017cech,bubenik2021applied}. 

\subsection{Closure spaces}
\label{section:definitions}
We start by recalling basic definitions and facts about closure spaces.

\begin{definition}
\label{def:closure_spaces}
For a set $X$, a function $c:\mathcal{P}(X)\to \mathcal{P}(X)$ is called a \emph{closure operation} (or simply \emph{closure}) for $X$ if the following conditions are met:
\begin{itemize}
\item[1)] $c(\emptyset)=\emptyset$,
\item[2)] $A\subset c(A)$ for all $A\subset X$,
\item[3)] $c(A\cup B)=c(A)\cup c(B)$ for all $A,B\subset X$.
\end{itemize}
An ordered pair $(X,c)$ where $X$ is a set and $c$ a closure for $X$ is called a \emph{(\v{C}ech) closure space}. Elements of $X$ are called \emph{points}.
\end{definition}

Let $(X,c)$ be a closure space. From the definition, one can easily show the following.

\begin{lemma}
\label{lemma:closures_are_monotone} 
If $A\subset B\subset X$, then $c(A)\subset c(B)$. 
\end{lemma}

 A subset $A\subset X$ is \emph{closed} if $c(A)=A$. A subset $A\subset X$ is \emph{open} if $X-A$ is closed.

\begin{definition}
\label{def:interior}
A closure space $(X,c)$ has an associated \emph{interior operation} for $X$, $i_c:\mathcal{P}(X)\to \mathcal{P}(X)$ defined by 
\[i_c(A):=X-c(X-A).\]
From the definitions, one can check the following:
\begin{itemize}
\item[1)] $i_c(X)=X$,
\item[2)] For all $A\subset X$, $i_c(A)\subset A$,
\item[3)] For all $A,B\subset X$, $i_c(A\cap B)=i_c(A)\cap i_c(B)$.
\end{itemize}

Conversely, if $i:\mathcal{P}(X)\to \mathcal{P}(X)$ is a function satisfying the 3 conditions in \cref{def:interior}, we can define $c_{i}:\mathcal{P}(X)\to \mathcal{P}(X)$ by 
\[c_{i}(A):=X-i(X-A).\]
As it turns out, $c_{i}$ is a closure operation for $X$. Often times we will write just $i$, instead of $i_c$ when the closure operation $c$ is clear from context. Similarly, we will just write $c$, instead of $c_i$ when the interior operation $i$ is clear from context.

\begin{proposition}{\cite[14.A.12]{vcech1966topological}}
\label{proposition:open_if_equal_to_interior}
 $A\subset X$ is open iff $i(A)=A$.
\end{proposition}
\end{definition}

\begin{definition}{\cite[Definition 14.B.1]{vcech1966topological}}
\label{def:neighborhood}
 Let $A\subset X$. A subset $B\subset X$ is a \emph{neighborhood} of $A$ if $A\subset i(B)$. If $A=\{x\}$, we say $B$ is a neighborhood of $x\in X$. The \emph{neighborhood system of $A$} is the collection of all neighborhoods of $A$.
\end{definition}

\begin{definition}
\label{def:coarser_finer_closure_operations}
 Suppose $c_1$ and $c_2$ are two closure operations for $X$. We say $c_1$ is \emph{coarser} than $c_2$  and $c_2$ is \emph{finer} than $c_1$ if $c_2(A)\subset c_1(A)$ for all $A\subset X$.
\end{definition}

\begin{example}
\label{example:discrete_and_indiscrete_closures}
For a set $X$, the identity map $\mathbf{1}_{\mathcal{P}(X)}:\mathcal{P}(X)\to \mathcal{P}(X)$ is a closure operation for $X$. It is called the \emph{discrete closure for} $X$, and it is the minimal closure operation we can place on $X$. We also have the \emph{indiscrete closure for} $X$, defined by $A\mapsto X$ for $A\neq \emptyset$ and $\emptyset\mapsto\emptyset$, which is also the maximal closure we can place on $X$.
\end{example}

\begin{definition}
\label{def:topological_spaces}
A \emph{topological closure operation} for a set $X$ is a closure $c$ for $X$ satisfying the additional requirement that $c(c(A))=c(A)$, for all $A\subset X$.
A closure operation like this is also called a \emph{Kuratowski closure operation} (see \cite[Remark 2.19]{rieser2017cech}). $(X,c)$ is called a \emph{topological space} (Kuratowski closure space) if $c$ is \emph{topological}.
\end{definition}

\begin{example}{\cite[Example 2.17]{rieser2017cech}}
\label{example:topological_space_is_closure_space}
Let $(X,\tau)$ be a topological space. Then the operation of taking the closure of a subset $A\subseteq X$ defines a closure operation for $X$, which is topological.
\end{example}

\begin{remark}{\cite[Example 2.17]{rieser2017cech} }
\label{remark:topological_spaces_are_closure_spaces}
\cref{def:topological_spaces} is equivalent to the standard definition of a topological space in the sense that the collection of open sets obtained from a topological closure form a topology and that the closure of a set is the usual closure in this topology.
\end{remark}

Hence, every topological space is a closure space. On the other hand, closure spaces where the closure operation does not come from taking topological closures induced by a topology one can put on the space in question are common (\cref{example:closure_space_not_a_topological_space}).

\begin{example}
\label{example:closure_space_not_a_topological_space}
Give $\mathbb{R}^n$ the euclidean metric $d$. Let $r>0$ and define $c_r$ be the closure operation for $\mathbb{R}^n$ by $c_r(A):=\{x\in \mathbb{R}^n\, |\, \text{dist}(x, A)\le r\}$ for all $A\subset \mathbb{R}^n$,  where $\text{dist}(x,A):=\inf_{y\in A}d(x,y)$. Let $U=\{x\in \mathbb{R}^n\,|\,d(x,0)\le 1\}$. Then $c_r(c_r(U))\neq c_r(U)$. Hence, there is no topology we can give $\mathbb{R}^n$ so that $c_r$ corresponds to a closure operation induced by said topology.
\end{example}

\begin{definition}[{\cite[Definition 17.A.17]{vcech1966topological} and \cite[Definition 2.6]{rieser2021grothendieck}}]
\label{def:interior_cover}

\begin{enumerate}
\item  A \emph{cover} of $(X,c)$ is a family of subsets of $X$, $\mathcal{U}=\{U_{j}\}_{j\in J}$, such that $X=\bigcup_{j\in J}U_j$. 
\item We say $\mathcal{U}$ is an \emph{interior cover} of $(X,c)$ if every point $x\in X$ has a neighborhood in $\mathcal{U}$. That is, $X=\bigcup_{j\in J}i(U_j)$.
\item We say $\mathcal{U}$ is an \emph{open cover} if every $U_j$ is open and a \emph{closed cover} if every $U_j$ is closed. Note that any open cover of $(X,c)$ is an interior cover. 
%\item For any subset $A\subseteq X$, we say that a collection $\mathcal{U}=\{U_j\}_{j\in J}$ of subsets of $A$ is an \emph{$i$-cover} of $A$ if $A=\bigcup_{j\in J}U_j$ and $i(A)=\bigcup_{j\in J}i(U_j)$. \nm{finish definition later}
 \end{enumerate}
\end{definition}

\begin{definition}
\label{def:locally_finite}
A family $\{U_{\alpha}\, |\, \alpha\in A \}$ of subsets of a closure space $(X,c)$ is called \emph{locally finite} of each point $x\in X$ possesses a neighborhood intersecting only finitely many $U_{\alpha}$.
\end{definition}

\subsection{Continuous maps}

Unless otherwise specified, $(X,c)$ and $(Y,d)$ will be closure spaces whenever written.

\begin{definition}{\cite[Definition 16.A.1]{vcech1966topological}}
\label{def:continuous_functions}
 A map $f:(X,c)\to (Y,d)$ is \emph{continuous at} $x\in X$ if for all $A\subset X$ such that $x\in c(A)$ it follows that $f(x)\in d(f(A))$. If $f$ is continuous at every $x\in X$, $f$ is called \emph{continuous}. Equivalently, $f$ is continuous if for every $A\subset X$, $f(c(A))\subseteq d(f(A))$. A continuous map $f$ is called a \emph{homeomorphism} if $f$ is a bijection with a continuous inverse.
\end{definition}

 Observe that a closure $c_1$ for $X$ is coarser than $c_2$ for $X$ iff the identity map $\mathbf{1}_X:(X,c_2)\to (X,c_1)$ is continuous.

\begin{lemma}{\cite[Lemma 4.22]{rieser2017cech}}
\label{lemma:homeomorphism}
Let $f:(X,c)\to (Y,d)$ be a bijection. Then $f(c(A))=d(f(A))$ for all $A\subset X$ iff $f$ is a homeomorphism.
\end{lemma}

\begin{theorem}{\cite[Theorem 16.A.4 and Corollary 16.A.5]{vcech1966topological}}
\label{theorem:equivalent_definitions_of_continuity}
 A map $f:(X,c)\to (Y,d)$ is continuous at $x\in X$ iff for all neighborhoods $V\subset Y$ of $f(x)$, the inverse image $f^{-1}(V)$ is a neighborhood of $x$. Equivalently, $f$ is continuous at $x$ iff for each neighborhood $V\subset Y$ of $f(x)$, there exists a neighborhood $U\subset X$ of $x$ such that $f(U)\subset V$. 
\end{theorem}

\begin{proposition}{\cite[16.A.3]{vcech1966topological}}
\label{prop:composition_of_cont_maps}
Let $f:(X,c)\to (Y,d)$ and $g:(Y,d)\to (Z,e)$ be continuous maps of closure spaces. Then the composition $g\circ f:(X,c)\to (Z,e)$ is continuous.
\end{proposition}

\begin{theorem}{\cite[17.A.18]{vcech1966topological}}[Pasting Lemma]
\label{theorem:pasting_lemma_for_closure_spaces}
Let $\{U_{\alpha}\,|\,\alpha\in A\}$ be a locally finite closed cover of $(X,c)$. Let $f:(X,c)\to (Y,d)$ be a map of sets. If $f|_{U_{\alpha}}$ is continuous for all $\alpha\in A$, then $f$ is continuous.
\end{theorem}

\begin{theorem}{\cite[Theorem 16.A.10]{vcech1966topological}}
\label{theorem:continuous_maps_of_topological_spaces}
Suppose $(Y,d)$ is a topological space.  A map $f:(X,c)\to (Y,d)$ is continuous iff the inverse image of every open set is open. Equivalently, $f$ is continuous iff the inverse image of every closed set is closed.
\end{theorem}

\begin{definition}
\label{def:closure_space_category}
Denote the category with objects closure spaces and morphism continuous maps between closure spaces by $\cat{Cl}$ and by $\cat{Top}$, the full subcategory of with objects topological spaces.
\end{definition}

\begin{proposition}{\cite[16.B.1-16.B.3]{vcech1966topological}}
\label{prop:topological_modification}
For a closure space $(X,c)$, let $\tau (c):\mathcal{P}(X)\to \mathcal{P}(X)$ by 
\[\tau(c)(A):=\bigcap\{F\subseteq X\,|\, c(F)=F\text{ and } A\subseteq F\}\]
Then $\tau (c)$ is a topological closure operation. By construction, $\tau (c)$ is also the finest topological closure operation coarser than $c$, and is called \emph{the topological modification of $c$}.
\end{proposition}

\begin{proposition}{\cite[16.B.4]{vcech1966topological}}
\label{prop:adjoints_between_closures_and_topologies}
Let $(X,c)$ be a closure space and let $(Y,\tau)$ be a topological space. A map of sets $f:X\to Y$ is continuous as a closure space map $f:(X,c)\to (Y,\tau)$ iff the map $f:(X,\tau(c))\to (Y,\tau)$ is a continuous map of topological spaces. That is, there exists a natural bijection between the sets of morphisms
\[\cat{Cl}((X,c),(Y,\tau))\cong \cat{Top}((X,\tau(c)),(Y,\tau))\]
Equivalently, $\tau$ is the left adjoint to the inclusion functor $\iota:\cat{Top}\to \cat{Cl}$.
\end{proposition}

\begin{lemma}{\cite[Lemma 4.12]{rieser2017cech}}
\label{lemma:finite_cover_whose_image_is_contained_in_covering_system}
Let $(X,c_\tau)$ be a closure space with topological closure operation $c_{\tau}$ and suppose $X$ is compact with respect to the topology corresponding to $c_{\tau}$. Let $(Y,c)$ be a closure space with interior cover $\mathcal{C}$, and suppose that $f:(X,c_\tau)\to (Y,c)$ is continuous. Then there exists an open cover $\mathcal{U}$ of $(X,c_\tau)$ (finite if one wants) such that for all $U\in \mathcal{U}$, there is a $V\in \mathcal{C}$ such that $f(U)\subset V$.
\end{lemma}

\subsection{Examples}
We recall the notation introduced in \cite{bubenik2021applied} for examples of closure spaces that are used to define homology and homotopy theories of closure spaces. Let $m\ge 0$.

\begin{definition}
\label{def:intervals}
\begin{enumerate}
\item Let $I$  be the unit interval $[0,1]$ with its standard topology.
\item Let $J_{m,\bot}$ be the set $\{0,\dots, m\}$ with the discrete closure. 
\item Let $J_{m,\top}$ be the set $\{0,\ldots,m\}$ with the indiscrete closure.  
\item Let $J_m$ be the set $\{0,\dots ,m\}$ with the closure $c(i)=\{j\in \{0,\dots, m\}\,|\, |i-j|\le 1\}$. A special case is $J_1 = J_{1,\top}$.
\item Let $J_+$ be the set $\{0,1\}$ with the closure $c_+(0)=\{0,1\}$, $c_+(1)=\{1\}$. Let $J_{-}$ denote the set $\{0,1\}$ with the closure $c_-(0)=\{0\}$, $c_-(1)=\{0,1\}$.
\item For $0 \leq k \leq 2^m-1$, define a closure $c_k$ on the set $\{0,\ldots,m\}$ in the following manner. Consider the binary representation of $k$.
For $1 \leq i \leq m$, $i-1$ is contained in $c_k(i)$ iff the $i$th rightmost bit is $0$ and
$i$ is contained in $c_k(i-1)$ iff the $i$-th rightmost bit is $1$. Let $J_{m,k}$ denote this closure space. Note the special cases, $J_{1,1} = J_{+}$ and $J_{1,0} = J_{-}$.
\item Let $J_{m,\leq}$ be the set $\{0,1,\dots ,m\}$ with the closure $c(i)=\{j\,\mid\, i\le j\}$. This closure is topological with the open sets being the down-sets. A special case is $J_{1,\leq} = J_+$.
\end{enumerate}
\end{definition}

\subsection{Canonical operations}

There are two canonical products of closure spaces, the product closure and the inductive product closure. The category $\cat{Cl}$ is also complete and cocomplete and here we also discuss coproducts, coequalizers and equalizers of closure spaces.

\begin{definition}{\cite[Definition 17.C.1]{vcech1966topological}}
\label{def:product_closure}
Let $\{X_{\alpha},c_{\alpha}\}_{\alpha\in A}$ be a family of closure spaces and  let $X=\prod_{\alpha\in A}X_{\alpha}$. For each  $\alpha\in A$, let $\pi_{\alpha}:X\to X_{\alpha}$ be the projection map. For each $x\in X$, let $\mathcal{U}_x$ denote the collection of sets of the form
\begin{equation} \label{eq:product}
  \bigcap\{\pi_{\alpha}^{-1}(V_{\alpha})\,|\,\alpha\in F\},
\end{equation}
where $F\subset A$ is finite and $V_{\alpha}$ is a neighborhood of $\pi_{\alpha}(x)$ in $(X_{\alpha},c_{\alpha})$. The collection $\mathcal{U}_x$ %satisfies all the conditions in \cref{prop:filter_base_properties}.
is a \emph{filter base} (\cref{def:base_of_a_filter}) in $X$ and $x \in \bigcap \mathcal{U}_{x}$.
Thus, by \cref{theorem:filter_determines_a_local_base_for_a_closure} there is a unique closure $c$ for $X$ such that $\mathcal{U}_x$ is a \emph{local base} at $x$ (\cref{def:base_of_a_neighborhood}) in $(X,c)$. The closure $c$ is called the \emph{product closure} for $X$ and the pair $(X,c)$ is called the \emph{product closure space}. The product closure is also characterized by a universal property ( \cref{prop:product_closure}).
Given closure spaces $(X,c)$ and $(Y,d)$, their product will be denoted by $(X\times Y,c\times d)$. Instead of using all neighborhoods $V_{\alpha}$ of $\pi_{\alpha}(x)$ in \eqref{eq:product}, we can restrict to a local base at $\pi_{\alpha}(x)$ {\cite[\nopp 17.C.3]{vcech1966topological}}. In particular, if $x$ and $y$ have local bases $\mathcal{U}_x$ and $\mathcal{U}_y$, respectively, then $\mathcal{U}_{(x,y)} = \mathcal{U}_x \times \mathcal{U}_y$ is a local base at $(x,y)$.
\end{definition}

\begin{definition}{\cite[Definition 17.D.1]{vcech1966topological}}
\label{def:inductive_product_closure}
Given closure spaces $(X,c)$ and $(Y,d)$, consider the product set $X\times Y$. For $(x,y)\in X\times Y$, let $\mathcal{V}_{(x,y)}$ be the collection of all sets of the form
\begin{equation} \label{eq:inductive} 
(\{x\}\times V)\cup (U\times \{y\})  %\tag{*}
\end{equation}
where $V$ and $U$ are neighborhoods of $y$ in $(Y,d)$ and $x$ in $(X,c)$, respectively. Each $\mathcal{V}_{(x,y)}$
is a \emph{filter base} (\cref{def:base_of_a_filter}) in $X\times Y$ and $(x,y)\in \bigcap \mathcal{V}_{(x,y)}$.
Thus, by \cref{theorem:filter_determines_a_local_base_for_a_closure}, there exists a unique closure operation for $X\times Y$ such that $\mathcal{V}_{(x,y)}$ is a local base (\cref{def:base_of_a_neighborhood}) at $(x,y)$ for all $(x,y)\in X\times Y$.
The \emph{inductive product} of spaces $(X,c)$ and $(Y,d)$, denoted by $(X\times Y,c \boxdot d)$ is the set $X\times Y$ endowed with this closure operation. The closure operation $c \boxdot d$ is called the \emph{inductive product closure}.
The neighborhoods of the form \eqref{eq:inductive} are called \emph{canonical neighborhoods for the inductive product}, or \emph{canonical inductive neighborhoods}.
Instead of using all neighborhoods of $x$ and $y$ in \eqref{eq:inductive}, we can restrict to local bases $\mathcal{U}_x$ and $\mathcal{U}_y$ of $x$ and $y$, respectively {\cite[\nopp 17.C.3]{vcech1966topological}}. If $(X,c)$ and $(Y,d)$ are closure spaces, we will denote their inductive product by $(X,c)\boxdot (Y,d)$.
The sets in \eqref{eq:inductive} may be written as $\{(x',y') \in U \times V \,|\, x'=x \text{ or } y'=y\}$.
\end{definition}

\begin{proposition}{\cite[Theorem 17.D.2]{vcech1966topological}}
\label{prop:comparison_of_products}
The product closure is coarser than the inductive product closure.
\end{proposition}

\begin{theorem}{\cite[Theorem 17.C.6]{vcech1966topological}}
\label{prop:product_closure}
%\nm{Remove if not used.}
The projections of a product space to its coordinate spaces are  continuous. Furthermore, the product closure is the coarsest closure for the  product of underlying sets such that all projections are continuous.
\end{theorem}
\begin{lemma}
\label{lemma:ind_product_and_product_of_discrete_spaces}
%\nm{Remove if not used.}
Let $X$ be a closure space and let $Y$ be a discrete space. Then $X\times Y=X\boxdot Y$. 
\end{lemma}

\begin{proposition}{\cite[Proposition 17.C.11]{vcech1966topological}}
\label{prop:product_of_maps}
  Suppose we are given for each $a \in A$ closure spaces  $(X_a,c_{X_a})$ and
  $(Y_a,c_{Y_a})$ and a map of sets $f_a:X_a \to Y_a$.
  If for all $a \in A$, $f_a$ is continuous, then the mapping $f:(\prod_{a\in A}X_a,\prod_{a\in A}c_{X_a})\to (\prod_{a\in A}Y_a,\prod_{a\in A}c_{Y_a})$ defined by $\{x_a\}_{a\in A}\mapsto \{f_x(x_a)\}_{a\in A}$ is continuous. Conversely, if $f$ is continuous and $\prod_{a\in A}X_a\neq \emptyset$, then for all $a \in A$, $f_a$ is continuous.
\end{proposition}

\begin{proposition}{\cite[Proposition 2.30]{bubenik2021applied}}
\label{prop:inductive_product_of_maps}
If $f:(X_1,c_1)\to (Y_1,d_1)$ and $g:(X_2,c_2)\to (Y_2,d_2)$ are continuous maps, then so is the map $f\times g:(X_1,c_1)\boxdot (X_2,c_2)\to (Y_1,d_1)\boxdot (Y_2,d_2)$ defined by $f\times g(x_1,x_2):=(f(x_1),g(x_2))$.
\end{proposition}

\begin{definition}{\cite[Definition 17.B.1]{vcech1966topological}}
\label{def:coproducts}
Let $\{(X_i,c_i)\}_{i\in I}$ be a collection of closure spaces. The \emph{coproduct}
of $\{(X_i,c_i)\}_{i\in I}$ is the  disjoint union of sets $X=\coprod_i X_i$ with the closure $c$ defined by $c(\coprod_i A_i):=\coprod_i c_i(A_i)$ for all subsets $\coprod_i A_i$ of $X$.
\end{definition}

\begin{definition} \label{def:coequalizer}
  Given continuous maps
  $
  \begin{tikzcd}(X,c) \ar[r,shift left,"f"] \ar[r,shift right,"g"'] & (Y,d),
  \end{tikzcd}
  $
 the \emph{coequalizer of $f$ and $g$} consists of the closure space $(Q,c_Q)$ and a map $p:Y \to Q$ defined in the following manner. Let $Q$ be the quotient set $Y/\!\! \sim$, where $\sim$ is the equivalence relation given by $f(x) \sim g(x)$, $\forall x \in X$. Let $p:Y \to Q$ be the quotient map. For any $A \subset Q$, set $c_Q(A) = p(d(p^{-1}(A)))$.
\end{definition}

\begin{theorem}{\cite[Theorems 33.A.4 and 33.A.5]{vcech1966topological}} \label{cocomplete}
 The coproduct and coequalizer defined above are the categorical coproduct and coequalizer in the category $\cat{Cl}$ and hence $\cat{Cl}$ is cocomplete.
\end{theorem}

We thus have pushouts of closure spaces.

\begin{definition}
  The \emph{pushout} in $\cat{Cl}$ is defined in the following manner. Given the solid arrow diagram
    \begin{equation*}
      \begin{tikzcd}
        (A,c_A) \ar[r,"f"] \ar[d,"g"] & (X,c) \ar[d,dashed,"i"]\\
        (Y,d) \ar[r,dashed,"j"'] & (P,c_P)
      \end{tikzcd}
    \end{equation*}
   define $P = (X \amalg Y)/\!\! \sim$ where $f(a) \sim g(a)$ for all $a \in A$. Denote by $i,j$ the induced maps, and for any $B \subset P$, set $c_P(B) = i(c(i^{-1}(B))) \cup j(d(j^{-1}(B)))$.
\end{definition}

\begin{lemma}{\cite[Lemma 2.35]{bubenik2021applied}}
\label{lemma:generalized_intervals}
$J_{m,\bot}$, $J_m$ and $J_{m,k}$ are obtained by $m$-fold binary pushouts of $J_{1,\bot}$, $J_1$, and $J_{-},J_{+}$ respectively under $*$. 
\end{lemma}

\begin{lemma}{\cite[Lemma 2.36]{bubenik2021applied}}
\label{lemma:comparison_of_intervals}
%\nm{remove if never used.}
The identity maps $J_{m,k}\xrightarrow{\mathbf{1}} J_m$ are continuous for all $m\ge 0$ and $0\le k\le 2^m-1$.
% The neareast neighbor map $f:(I,\tau)\to J_{1}$ defined by $f(x)=0$ if $x<\frac{1}{2}$ and $f(x)=1$ if $x\ge \frac{1}{2}$ is also continuous.
The `round up' map $f_+:(I,\tau)\to J_{+}$ defined by $f(x)=0$ if $x<\frac{1}{2}$ and $f(x)=1$ if $x\ge \frac{1}{2}$
and the `round down' map $f_{-}:(I,\tau)\to J_{-}$ defined by $f(x)=0$ if $x\leq\frac{1}{2}$ and $f(x)=1$ if $x> \frac{1}{2}$
are continuous.
These may be combined to obtain continuous maps $f:([0,m],\tau)\to J_{m,k}$ for any $m\ge 0$ and $0\le k\le 2^m-1$.
Precomposing with the map $t \mapsto mt$, we obtain a continuous map $f:(I,\tau)\to J_{m,k}$.
\end{lemma}

Dually, we also have the following.

\begin{definition} \label{def:equalizer}
  Given continuous maps
  $
  \begin{tikzcd}(X,c) \ar[r,shift left,"f"] \ar[r,shift right,"g"'] & (Y,d),
  \end{tikzcd}$
  the \emph{equalizer of $f$ and $g$} consists of the closure space $(E,c_E)$ and map $i:E \to X$ defined in the following manner. Let $E$ be the subset $\{x \in X \ | \ f(x)=g(x)\}$ with $i$ the inclusion map.
  For any $A \subset E$, set $c_E(A) = c(A) \cap E$.
\end{definition}

\begin{theorem}{\cite[Theorems 32.A.4 and 32.A.10]{vcech1966topological}} \label{complete}
  The product and equalizer defined above are the categorical product and equalizer in the category $\cat{Cl}$ and hence $\cat{Cl}$ is complete.
\end{theorem}

\begin{proposition}\cite[Proposition 2.39]{bubenik2021applied}
\label{prop:limits_and_colimits_of_closure_spaces}
 Every limit of topological spaces in $\mathbf{Cl}$ is a topological space. On the other hand, colimits of topological spaces in $\mathbf{Cl}$ are not necessarily topological spaces.
\end{proposition}

\section{Closure spaces, simplicial and cubical sets}
\label{section:closure_spaces_and_simplicial_sets}
In this section we define various realization and nerves functors between closure spaces and simplicial and cubical sets. The construction follows a formalism developed in simplicial homotopy theory and can be summarized by the following classical results. 

\begin{theorem}{\cite[Theorem 6.5.8]{riehl2017category}}
\label{theorem:density}
For any small category $\mathbf{C}$, the identity functor defines the left Kan extension of the Yoneda embedding $y:\mathbf{C}\to \mathbf{Set}^{\mathbf{C}^{op}}$ along itself.
\begin{figure}[H]
\centering
\begin{tikzcd}
\mathbf{C}\arrow[rr,"y"]\arrow[dr,"y"] &&\mathbf{Set}^{\mathbf{C}^{op}}\\
&\mathbf{Set}^{\mathbf{C}^{op}}\arrow[ur,dashed,"1\cong \text{Lan}_yy"]
\end{tikzcd}
\end{figure}
\end{theorem}

\begin{proposition}{\cite[Remark 6.5.9]{riehl2017category}}
\label{prop:defining_geometric_realization}
Any functor $F:\mathbf{C}\to \mathbf{E}$ whose domain is small and whose codomain is locally small and cocomplete admits a left Kan extension along the Yoneda embedding $y:\mathbf{C}\to \mathbf{Set}^{\mathbf{C}^{op}}$.
\begin{figure}[H]
\centering
\begin{tikzcd}
\mathbf{C}\arrow[rr,"F"]\arrow[dr,"y"] &&\mathbf{E}\\
&\mathbf{Set}^{\mathbf{C}^{op}}\arrow[ur,dashed,"\text{Lan}_yF"]
\end{tikzcd}
\end{figure}
\end{proposition}

\begin{proposition}{\cite[Exercise 6.5.v]{riehl2017category}}
\label{prop:riehl_exercise}
Let $\Delta$ be the category of finite ordinals and order preserving maps between them (\cref{def:simplex_category}). Applying the construction in \cref{prop:defining_geometric_realization} to the functor $|\Delta|:\Delta\to \mathbf{Top}$ that sends the ordinal $[n]$ to the topological $n$-simplex 
\[|\Delta^n|:=\{(x_0,\dots ,x_n)\in \mathbb{R}^{n+1}\,|\,\sum_i x_i=1,x_i\ge 0\}\]
gives us a functor $|-|:\mathbf{Set}^{\Delta^{op}}\to \mathbf{Top}$, called the \emph{geometric realization functor}. This functor has a right adjoint, called the \emph{total singular complex functor}, $\mathscr{S}:\mathbf{Top}\to \mathbf{Set}^{\Delta^{op}}$ which is used to define singular homology of topological spaces.
\end{proposition}

Using the results of \cref{complete,cocomplete}
and \cref{prop:defining_geometric_realization} along with the idea in \cref{prop:riehl_exercise} we define several analogues of geometric realization and total singular complex functors for closure spaces in \cref{section:simplicial_realizations,section:cubical_realizations}. We then use simplicial and cubical Dold-Kan correspondences to define several homology theories for closure spaces in \cref{section:homology_theories}.

 \subsection{Simplex category and simplicial sets}
\label{section:simplicial_sets}
Here we recall the definitions of simplicial sets and the simplex category.

\begin{definition}{\cite[Chapter 1]{goerss2009simplicial}}
\label{def:simplex_category}
\emph{The simplex category} $\Delta$ is the category whose objects are finite totally ordered sets (finite ordinals), $[n]=\{0,\dots , n\}$, and whose morphisms are order preserving functions between ordinals. Observe that $\Delta$ has two obvious subcategories; the category $\Delta_+$ consisting of all injective order preserving maps and the category $\Delta_-$ consisting of all surjective order preserving maps. Furthermore, every morphism in $\Delta$ can be written uniquely as a composition of a morphism in $\Delta_-$ followed by a morphism in $\Delta_+$. Even more, all morphisms in $\Delta$ are+ generated by the morphisms $d^i:[n-1]\to [n]$ for $n\ge 1$ and $0\le i\le n$ where the image of $d^i$ does not include $i$ and the morphisms $s^i:[n]\to [n-1]$ for $n\ge 1$ and $0\le i\le n-1$, where $s^i$ identifies $i$ and $i+1$. There exist the following relations between these maps, called \emph{cosimplicial identities}:
\begin{gather*}
d^jd^i=d^id^{j-1}\, (i<j)\\
s^jd^i=d^is^{j-1}\, (i<j)\\
s^jd^i=\text{id}\, (i=j,j+1)\\
s^jd^i=d^{i-1}s^j\, (i>j+1)\\
s^js^i=s^{i-1}s^j\, (i>j)
\end{gather*}
\end{definition}

\begin{definition}
\label{def:simplicial_set}
A \emph{simplicial set} is a contravariant functor $X:\Delta\to \mathbf{Set}$. We denote $X([n])$ by $X_n$ and $X(d^i)$ by $d_i$ and $X(s^i)$ by $s_i$ for simplicity. We call an element $x$ of $X_n$ an $n$-\emph{simplex}. Morphisms of simplicial sets are natural transformations between functors, called \emph{simplicial maps}, and we denote the category of simplicial sets by $\mathbf{sSet}$. We denote by $\Delta^n$ the simplicial set defined by $\Delta^n:=\Delta(-,[n])$. The Yoneda Lemma implies that simplicial maps $\Delta^n\to X$ classify the $n$-simplices of $X$, i.e., there is a natural set-theoretic bijection 
\[\mathbf{sSet}(\Delta^n,X)\cong X_n\]
\end{definition}

\subsection{Box category and cubical sets}
\label{section:cubical_sets}
Here we recall the definitions of cubical sets and the box category. 

\begin{definition}
\label{def:cube_category}
Let $\mathbf{1}=[1]$ be the second ordinal with the standard partial order. Note that $\mathbf{1}$ is also a category whose objects are $0$ and $1$ and a unique morphism $0\to 1$, besides the identity morphisms. Denote by $\mathbf{1}^n$ the $n$-fold categorical product of $\mathbf{1}$ and by $\mathbf{1}^0$ the category with one object and one morphism. \emph{The box category} $\Box$ is the category whose objects are $\mathbf{1}^n$, for $n\in \mathbb{N}$. The morphisms are given by order preserving set theoretic maps. It turns out that every morphism in $\Box$ can be written uniquely as a composition of \emph{face} and \emph{degeneracy} functors. More specifically, each $(i,\epsilon)\in \{1,\dots ,n\}\times \mathbf{1}$ determines a unique face functor $d^{(i,\epsilon)}:\mathbf{1}^{n-1}\to \mathbf{1}^n$, defined by
\[d^{i,\epsilon}(a_1,a_2,\dots, a_{n-1}):=(a_1,\dots, a_{i-1},\epsilon,a_i,\dots, a_{n-1})\]
and each $1\le j\le n$ determines a unique degeneracy functor $s^j:\mathbf{1}^n\to \mathbf{1}^{n-1}$ defined by 
\[s^j(a_1,\dots ,a_n):=(a_1,\dots ,a_{j-1},a_{j+1},\dots ,a_n)\]
Furthermore, there exist the following relations among the face and degeneracy functors:
\begin{gather*}
s^js^i=s^{i}s^{j+1}\, (i\le j)\\
s^jd^{(j,\epsilon)}=\text{id}\\
s^jd^{(i,\epsilon)}=d^{(i,\epsilon)}s^{j-1}, (i<j)\\
s^jd^{(i+1,\epsilon)}=d^{(i,\epsilon)}s^j, (i\ge j)
\end{gather*}
For more details on this, see for example \cite[Section 1]{jardine2002cubical}.
\end{definition}

\begin{definition}
\label{def:cubical_set}
A \emph{cubical set} is a contravariant functor $X:\Box\to \mathbf{Set}$. Morphisms of cubical sets are natural transformations between functors and we denote the category of cubical sets by $\mathbf{cSet}$. We denote $X(\mathbf{1}^n)$ by $X_n$ and $X(d^{i,\epsilon})$ by $d_{(i,\epsilon)}$ and $X(s^i)$ by $s_i$ for simplicity. We call $X_n$ the set of $n$-\emph{cells} of $X$. We denote by $\Box^n$ the \emph{standard} $n$-\emph{cell}, the cubical set defined by $\Box^n:=\Box(-,\mathbf{1}^n)$. The Yoneda Lemma implies that a morphisms of cubical sets $\Box^n\to X$ classify the $n$-cells of $X$, i.e., there is a natural set theoretic bijection
\[\mathbf{cSet}(\Box^n,X)\cong X_n\]
\end{definition}

\begin{definition}
\label{def:simplex_category}
Let $X$ be a simplicial set. Let $\Delta\downarrow X$ be the \emph{simplex category} of $X$. The objects of $\Delta\downarrow X$ are simplicial set morphisms $\sigma:\Delta^n\to X$, called \emph{simplices} of $X$. A morphism of $\Delta\downarrow X$ is a commutative diagram of simplicial maps
\begin{figure}[H]
\centering
\begin{tikzcd}
\Delta^n\arrow[dr,"\sigma"]\arrow[dd,"\theta"]\\
& X\\
\Delta^m\arrow[ur,"\tau"]
\end{tikzcd}
\end{figure}
\end{definition}

\begin{definition}
\label{def:cell_category}
Let $X$ be a cubical set. Let $\Box\downarrow X$ be \emph{cell category} of $X$. The objects of $\Box\downarrow X$ are cubical set morphisms $\sigma:\Box^n\to X$ called $n$-\emph{cells} of $X$. A morphism of $\Box\downarrow X$ is a commutative diagram of cubical set maps
\begin{figure}[H]
\centering
\begin{tikzcd}
\Box^n\arrow[dr,"\sigma"]\arrow[dd,"\theta"]\\
& X\\
\Box^m\arrow[ur,"\tau"]
\end{tikzcd}
\end{figure}
\end{definition}

\begin{lemma}{\cite[Lemma 2.1]{goerss2009simplicial}}
\label{lemma:simplicial_and_cubical_sets_as_colimits}
Let $X$ be a simplicial set. There is a natural simplicial isomorphism $X\cong\colim_{\Delta^n\to X}\Delta^n$ where the colimit is taken over $\Delta\downarrow X$. Let $Y$ be a cubical set. There is a natural isomorphism of cubical sets $Y\cong\colim_{\Box^n\to Y}\Box^n$ where the colimit is taken over $\Box\downarrow Y$.
\end{lemma}

\begin{proof}
Any functor $\mathbf{C}\to \mathbf{Set}$, where $\mathbf{C}$ is a small category, is a colimit of representable functors.
\end{proof}

\subsection{Simplicial realization and nerve functors}
\label{section:simplicial_realizations}
Here we define the geometric realization of simplicial sets and the singular simplicial sets of closure spaces. To be in the spirit of the theory of simplicial sets, we introduce the following notation for some closure spaces. Let $|\Delta|^I=|\Delta^n|$ be the standard $n$-simplex in $\mathbb{R}^{n+1}$, let $\disimp{n}=J_{n,\le}$ and let $\indsimp{n}=J_{n,\top}$ (\cref{def:intervals}). Let $J\in \{I,J_+,J_1\}$ and let $\gsimp{n}$ be the general symbol denoting one of these.

\begin{definition}[Simplicial realizations]
\label{def:geometric_realization}
Let $X$ be a simplicial set. We define the \emph{$J$ realization of $X$} to be 
\[|X|^{J}:=\colim_{\Delta^n\to X}\gsimp{n}\,,\]
where the colimit is taken over $\Delta\downarrow X$.
\end{definition}

\begin{remark}
\label{remark:geometric_realization}
A standard result is that the geometric realization of a simplicial set is a CW complex, when the colimit is calculated in the category of topological spaces. See for example \cite[Proposition 2.3]{goerss2009simplicial}. However, due to \cref{prop:limits_and_colimits_of_closure_spaces} this might no longer be the case.
\end{remark}

\begin{example}
\label{example:geometric_realization}
The $J$ realization of $\Delta^n$ is naturally homeomorphic to $\gsimp{n}$, justifying the notation. This is because $\Delta^n$ is the initial object in $\Delta\downarrow \Delta^n$. Note that in all cases, this is in fact a topological space. 
\begin{enumerate}
\item $|\Delta^n|^I$ is the standard $n$-simplex in $\mathbb{R}^n$. 
\item $\indsimp{n}$ is the topological space $\{0,\dots, n\}$ with the indiscrete topology. 
\item $\disimp{n}$ is the topological space $\{0,\dots, n\}$ where the open sets are the down-sets.
\end{enumerate} 
 However, due to \cref{prop:limits_and_colimits_of_closure_spaces} we cannot expect the $J$ realization of an arbitrary simplicial set $X$ to always be a topological space.
\end{example}

\begin{definition}[Simplicial nerve functors]
\label{def:singular_set}
Let $Y$ be a closure space. The \emph{$J$ nerve of $Y$} is the simplicial set $\mathscr{S}^{J}(Y):\Delta^{\text{op}}\to \cat{Set}$ that assigns to $[n]$ the set $\cat{Cl}(\gsimp{n},Y)$, the set of all continuous map from $\gsimp{n}$ to  $Y$.
\end{definition}

\begin{theorem}
\label{theorem:geometric_realization_adjunction}
If $X$ is a simplicial set and $Y$ a closure space, then there is a canonical bijection
\[\cat{Cl}(|X|^{J},Y)\cong \cat{sSet}(X,\mathscr{S}^{J}(Y))\],
i.e., $|-|^{J}$ is the left adjoint of $\mathscr{S}^{J}$.
\end{theorem}

\begin{proof}
We have the following natural isomorphisms
\begin{gather*}
\cat{Cl}(|X|^{J},Y)=\cat{Cl}(\colim\limits_{\Delta^n\to X}\gsimp{n},Y)\cong\\
\cong \lim\limits_{\Delta^n\to X}\cat{Cl}(\gsimp{n},Y)\cong\lim\limits_{\Delta^n\to X}\mathscr{S}^{J}(Y)_n\cong\\
\cong\lim\limits_{\Delta^n\to X}\cat{sSet}(\Delta^n,\mathscr{S}^{J}(Y))\cong \cat{sSet}(\colim_{\Delta^n\to X}\Delta^n,\mathscr{S}^{J}(Y))\cong\\
\cong \cat{sSet}(X,\mathscr{S}^{J}(Y))
\qedhere
\end{gather*}
\end{proof}

\subsection{Cubical realization and nerve functors}
\label{section:cubical_realizations}
Here we define the realization of cubical sets and discrete singular cubical sets of closure spaces.  Let $J\in \{I,J_+,J_1\}$, $\otimes \in\{\times,\boxdot\}$ and let $\gcube{n}=J^{\otimes n}$, the $n$-fold $\otimes$ product of $J$ with itself. By convention, if $n=0$, $\gcube{n}$ is the one-point space.

\begin{definition}[Cubical realizations]
\label{def:discrete_realization}
Let $X$ be a cubical set. We define the \emph{$(J,\otimes)$ realization of $X$ to be} 
\[|X|^{(J,\otimes)}:=\colim_{\Box^n\to X}\gcube{n},\]
where the colimit is taken over $\Box\downarrow X$.
\end{definition}

\begin{example}
\label{example:discrete_realization}
Analogous to \cref{example:geometric_realization}, the $(J,\otimes)$ realization of $\Box^n$ is $\gcube{n}$, as $\Box^n$ is the initial object in $\Box\downarrow X$. When, $\otimes=\times$, $\gcube{n}$ is a topological space in all cases.
\begin{enumerate}
\item For $J=I$, $\otimes=\times$, $\gcube{n}$ is the standard unit $n$-cube in $\mathbb{R}^n$.
\item For $J=J_1$, $\otimes=\times$, $\gcube{n}$ is the topological space $\{0,1\}^n$ with the indiscrete topology.
\item For $J=J_+$, $\otimes=\times$, $\gcube{n}$ is the topological space $\{0,1\}^n$ where the open sets are the down-sets with respect to the product partial order.
\end{enumerate}
 However, due to \cref{prop:limits_and_colimits_of_closure_spaces} we cannot expect the $(J,\times)$ realization of an arbitrary cubical set $X$ to always be a topological space. Furthermore, when $\otimes=\boxdot$, $\gcube{n}$ is never a topological space for $n\ge 2$, as discussed below.
 \begin{enumerate}
 \item Let $n\ge 2$ and consider $\itopcube{n}$. Let $A=(0,1)^n\subseteq \itopcube{n}$. Let us denote the closure on $\itopcube{n}$ by $c$. Then, by \cref{def:inductive_product_closure} it follows that $c(A)$ is the set $I^n$ minus the points whose all coordinates are either $0$'s or $1$'s. Furthermore, $c(c(A))=I^n$ and thus $c$ is not idempotent and hence not topological.
 \item Let $n\ge 2$ and consider $\iindcube{n}$. Let us denote the closure on $\iindcube{n}$ by $c$. Consider $(0,0,\dots,0)\in \{0,1\}^n$. Then, by \cref{def:inductive_product_closure} it follows that $c(0,0,\dots, 0)=\{(0,0,\dots ,0),(1,0,0,\dots ,0),(0,1,0,\dots ,0),\dots ,(0,0,\dots ,1)\}$. However, $c(1,0,0,\dots ,0)=\{(0,0,\dots ,0),(1,0,0,\dots, 0),(1,1,0,\dots ,0), (1,0,1,0,\dots ,0),\dots (1,0,0,\dots , 0,1)\}$ and therefore $c(c(0,0,\dots ,0))\neq c(0,0,\dots ,0)$ and thus $c$ is not idempotent and hence not topological. 
 \item  Let $n\ge 2$ and consider $\idicube{n}$. Let us denote the closure on $\idicube{n}$ by $c$. Consider $(0,0,\dots,0)\in \{0,1\}^n$. Then, by \cref{def:inductive_product_closure} it follows that $c(0,0,\dots, 0)=\{(0,0,\dots ,0),(1,0,0,\dots ,0),(0,1,0,\dots ,0),\dots ,(0,0,\dots ,1)\}$. However, $c(1,0,0,\dots ,0)=\{(1,0,0,\dots, 0),(1,1,0,\dots ,0), (1,0,1,0,\dots ,0),\dots (1,0,0,\dots , 0,1)\}$ and therefore 
 \newline
 $c(c(0,0,\dots ,0))\neq c(0,0,\dots ,0)$ and thus $c$ is not idempotent and hence not topological. 
 \end{enumerate}
\end{example}

\begin{definition}[Cubical nerve functors]
\label{def:discrete_singular_set}
Let $Y$ be a closure space. The \emph{$(J,\otimes)$ singular set} $\mathscr{C}^{(J,\otimes)}(Y)$ is a functor $\Box^{\text{op}}\to \mathbf{Set}$ that assigns to $\mathbf{1}^n$ the set $\mathbf{Cl}(\gcube{n},Y)$, the set of all continuous maps from $\gcube{n}$ to $Y$. 
\end{definition}

\begin{theorem}
\label{theorem:discrete_realization_adjunction}
If $X$ is a cubical set and $Y$ a closure space, then there is a canonical bijection
\[\mathbf{Cl}(|X|^{(J,\otimes)},Y)\cong \mathbf{cSet}(X,\mathscr{C}^{(J,\otimes)}(Y))\],
i.e., $|-|^{(J,\otimes)}$ is the left adjoint of $\mathscr{C}^{(J,\otimes)}$.
\end{theorem}

\begin{proof}
We have the following natural isomorphisms
\begin{gather*}
\mathbf{Cl}(|X|^{(J,\otimes)},Y)=\mathbf{Cl}(\colim\limits_{\Box^n\to X}\gcube{n},Y)\cong\\
\cong \lim\limits_{\Box^n\to X}\mathbf{Cl}(\gcube{n},Y)\cong\lim\limits_{\Box^n\to X}\mathscr{C}^{(J,\otimes)}(Y)_n\cong\\
\cong\lim\limits_{\Box^n\to X}\mathbf{cSet}(\Box^n,\mathscr{C}^{(J,\otimes)}(Y))\cong \mathbf{cSet}(\colim_{\Box^n\to X}\Box^n,\mathscr{C}^{(J,\otimes)}(Y))\cong\\
\cong \mathbf{cSet}(X,\mathscr{C}^{(J,\otimes)}(Y))
\qedhere
\end{gather*}
\end{proof}

\subsection{Fibrant objects}
\label{section:fibrant_objects}
Here we define fibrant simplicial and cubical sets and argue that our singular simplicial and cubical sets defined in previously are fibrant.

\begin{definition}
\label{def:horn}
The $k$-th \emph{horn} $\Lambda^k_n$ is the subobject of $\Delta^n$ generated by all $d_i(\mathbf{1}_n)$ for $i\neq k$ where $\mathbf{1}_n\in  \mathbf{sSet}(\Delta^n,\Delta^n)$ is the identity morphism.
\end{definition}

\begin{definition}
\label{def:Kan_condition}
A simplicial set $X$ satisfies the \emph{Kan condition} if any morphism of simplicial sets $\Lambda_k^n\to X$ can be extended to a morphism of simplicial sets $\Delta^n\to X$.
\end{definition}

\begin{theorem}{\cite[Theorem 3.4]{moore1958semi}}
\label{theorem:simplicial_groups_are_fibrant}
The simplicial set underlying any simplicial group (by forgetting the group structure) is a Kan complex.
\end{theorem}

\begin{corollary}
\label{corollary:singular_sets_are_fibrant}
Let $X$ be a closure space. Then the simplicial sets $\mathscr{S}^{J}(X)$ are Kan complexes. 
\end{corollary}

\begin{definition}{\cite[Definition 1.2]{tonks1992cubical}}
\label{def:cubical_sets_with_connections}
A \emph{cubical set with connections} is a cubical set $X$ together with, for $n\ge 1$ and each $i\in [n]$, two additional maps (called \emph{connections}) $\Gamma_i^{\epsilon}:X_n\to X_{n+1}$, for $\epsilon\in \{0,1\}$ such that, for $\epsilon,\delta\in \{0,1\}$ and $i,j\in [n]$, the following relations are satisfied:
\begin{align*}
\Gamma_i^{\epsilon}\Gamma_j^{\delta}&=\Gamma^{\delta}_{j+1}\Gamma_i^{\epsilon}\, \text{ if }i\le j\\
\Gamma_i^{\epsilon}s_j&=\begin{cases}
s_{j+1}\Gamma_i^{\epsilon} & \text{ if } i<j\\
s_j\Gamma_{i-1}^{\epsilon} & \text{ if } i>j\\
s_i^2=s_{i+1}s_i & \text{ if } i=j
\end{cases}\\
d_{(i,\epsilon)}\Gamma_j^{\delta}&=\begin{cases}
\Gamma_{j-1}^{\delta}d_{(i,\epsilon)}& \text{ if } i<j\\
\Gamma_j^{\delta}d_{(i-1,\epsilon)}& \text{ if } i>j+1\\
\text{id} & \text{ if } i=j,j+1,\epsilon=\delta\\
s_id_{(i,\epsilon)}& \text{ if } i=j,j+1,\epsilon\neq \delta
\end{cases}
\end{align*}
We denote the category of cubical sets with connections by $\mathbf{ccSet}$.
\end{definition}

\begin{definition}{\cite[Example 7]{jardine2002cubical}}
\label{def:cubical_horn}
The cubical set $\sqcap^n_{\epsilon,i}$ is the subobject of $\Box^n$ which is generated by all faces $d^{j,\delta}:\Box^{n-1}\subseteq \Box^n$ except for $d^{i,\epsilon}:\Box^{n-1}\to \Box^n$.
\end{definition}

\begin{definition}
\label{def:fibrant_cubical_set}
A cubical set $X$ is a \emph{Kan cubical complex} or \emph{fibrant} if any cubical set morphism $\sqcap^n_{\epsilon,i}\to X$ can be extended to a cubical set morphism $\Box^n\to X$.
\end{definition}

\begin{definition}{\cite[Definition 1.4]{tonks1992cubical}}
\label{def:cubical_group_with_connections}
A \emph{cubical group with connections $X$} is a cubical set $X$ with a group structure on each $X_n$ and such that each all the maps $d_{i,\epsilon}$, $s_i$ and $\Gamma^{\epsilon}_i$ is a group homomorphism.
\end{definition}

\begin{example}
\label{example:cubical_group_with_connections}
Let $X$ be a closure space and consider the cubical sets $\mathscr{C}^{(J,\otimes)}(X)$. Define the connection maps $\Gamma^{\epsilon}_{i}$ for $i\in [n]$ and $\epsilon\in \{0,1\}$ by 
\[
\Gamma_i^{\epsilon}\sigma(a_1,\dots ,a_{n+1}):=\sigma(a_1,\dots,a_{i-1},m_{\epsilon}(a_i,a_{i+1}),a_{i+2},\dots a_{n+1}),
\]
where 
\[m_{\epsilon}(x,y):=\begin{cases}
\min(x,y) & \text{ if } \epsilon=1\\
\max(x,y) & \text{ if } \epsilon=0
\end{cases}\]
The same arguments as in \cite[Lemma 2.3]{barcelo2018homology} show that $\Gamma_i^{\epsilon}$ are connections for the cubical sets $\mathscr{C}^{(J,\otimes)}(X )$. 
\end{example}

\begin{theorem}{\cite[Theorem 2.1]{tonks1992cubical}}
\label{theorem:cubical_groups_are_fibrant}
All cubical groups with connections are fibrant cubical sets (forgetting the group structure).
\end{theorem}

\begin{corollary}
\label{corollary:fibrant_cubical_sets}
Let $X$ be a closure space. Then the cubical sets $\mathscr{C}^{(J,\otimes)}(X)$ are fibrant.
\end{corollary}

\subsection{Dold-Kan correspondences}
\label{section:homology_theories}
Here, we use simplicial and cubical Dold-Kan correspondence to define various cubical and simplicial singular homology theories of closure spaces. These homology groups will be studied throughout the rest of the paper with the aim of showing that they satisfy analogues of standard results from algebraic topology.

\begin{definition}
\label{def:simplicial_abelian_group}
A \emph{simplicial abelian group} is a functor $X:\Delta^{\text{op}}\to \mathbf{Ab}$. The category of simplicial abelian groups will be denoted by $\mathbf{sAb}$.
\end{definition}

There is a functor $\mathbb{Z}:\mathbf{sSet}\to \mathbf{sAb}$ that assigns to each simplicial set $X$ the simplicial abelian group $\mathbb{Z}X$ where each $\mathbb{Z}X_n$ is the free abelian group on $X_n$. $\mathbb{Z}$ is the left adjoint to the forgetfull functor $U:\mathbf{sAb}\to \mathbf{sSet}$. 

To each simplicial abelian group $A$, there is an associated chain complex of abelian groups, called its \emph{Moore complex}, also denoted by $A$. where the boundary map $\partial_n:A_n\to A_{n-1}$ is defined by $\partial_n:=\sum_i(-1)^id_i$. Given a simplicial abelian group $A$, we can also define its \emph{normalized chain complex} $NA$, which is a chain complex defined by $NA_n:=\bigcap_{i=0}^{n-1}\ker (d_i)\subseteq A_n$. The maps $NA_n\xrightarrow{\sum_i(-1)^id_i}NA_{n-1}$ form a chain complex. The assignment $N:\cat{sAb}\to \cat{Ch}_{\ge 0}(\cat{Ab})$, defined by $N(A)\mapsto NA$, is also functorial. Let $DA_n$ be the subgroup of $A_n$ generated by the degenerate simplices. The boundary map $\partial$ of the Moore complex associated to $A$ induces a homomorphism $\partial_n:A_n/DA_n\to A_{n-1}/DA_{n-1}$. The resulting complex is denoted by $A/D(A)$. There exists chain maps
\[NA\xrightarrow{i}A\xrightarrow{p}A/D(A),\]
where $i$ and $p$ are the obvious inclusion and projection maps, respectively.
 
\begin{theorem}{\cite[Chapter 3, Theorem 2.1]{goerss2009simplicial}}
\label{theorem:normalized_simplicial_complex}
The composite $NA\xrightarrow{pi}A/D(A)$ is an isomorphism of chain complexes.
\end{theorem}

 Furthermore, there is an equivalence of categories $\mathbf{sAb}$ and $\mathbf{Ch}_{\ge 0}(\mathbf{Ab})$, also known classically as the Dold-Kan correspondence.

\begin{theorem}{\cite[Chapter 3, Corollary 2.3]{goerss2009simplicial}}[Simplicial Dold-Kan]
\label{theorem:simplicial_dold_kan}
There is an equivalence of categories $N:\cat{sAb}\leftrightarrows \cat{Ch}_{\ge 0}(\cat{Ab}):\Gamma$, where $N$ is the normalization functor.
\end{theorem}

\begin{theorem}{\cite[Chapter 3, Theorem 2.4]{goerss2009simplicial}}
\label{theorem:inclusion_of_normalized_simplicial_complex}
The inclusion $i:NA\to A$ of the normalized chain complex into the Moore complex, of a simplicial abelian group $A$, is a chain homotopy equivalence. The equivalence is natural with respect to simplicial abelian groups $A$.
\end{theorem}

\begin{definition}
\label{def:cubical_abelian_groups}
A \emph{cubical abelian group} is a functor $X:\Box^{\text{op}}\to \mathbf{Ab}$. The category of cubical abelian groups will be denoted by $\mathbf{cAb}$.
\end{definition}

There is a functor $\mathbb{Z}:\mathbf{cSet}\to \mathbf{cAb}$ that assigns to each cubical set $X$ the cubical abelian group $\mathbb{Z}X$ where each $\mathbb{Z}X_n$ is the free abelian group on $X_n$. $\mathbb{Z}$ is the left adjoint to the forgetfull functor $U:\mathbf{cAb}\to \mathbf{cSet}$. 

To each cubical abelian group $A$, there is an associated chain complex of abelian groups, called its \emph{unnormalized complex}, also denoted by $A$. where the boundary map $\partial_n:A_n\to A_{n-1}$ is defined by $\partial_n:=\sum_{i=1}^n (-1)^i(d_{i,0}-d_{i,1})$. Given a cubical abelian group $A$, we can also define its \emph{normalized chain complex} $NA$ in the following way:
\begin{gather*}
NA_0=A_0, NA_n:=\cap_{i=1}^n\ker (d_{i,1}),n>0,\\
\partial_n:=\sum_{i=1}^n(-1)^{i+1}d_{i,0}:NA_n\to NA_{n-1},n>0
\end{gather*}
The assignment $N:\cat{cAb}\to \cat{Ch}_{\ge 0}(\cat{Ab})$, defined by $N(A)\mapsto NA$, is also functorial. Let $DA_n$ be the subgroup of $A_n$ generated by the degenerate cubes. The boundary map $\partial$ of the unnormalized complex associated to $A$ induces a homomorphism $\partial_n:A_n/DA_n\to A_{n-1}/DA_{n-1}$. The resulting complex is denoted by $A/D(A)$. There exists chain maps
\[NA\xrightarrow{i}A\xrightarrow{p}A/D(A),\]
where $i$ and $p$ are the obvious inclusion and projection maps, respectively.
 
\begin{theorem}
\label{theorem:normalized_cubical_complex}
The composite $NA\xrightarrow{pi}A/D(A)$ is an isomorphism of chain complexes.
\end{theorem}

\begin{proof}
The proof is inspired by the proof of \cite[Chapter 3, Theorem 2.1]{goerss2009simplicial}, which covers the analogous statement in the case of simplicial abelian groups.
Let $N_jA_n:=\cap_{i=1}^j\ker (d_{i,1})\subseteq A_n$ and let $D_j(A_n)$  be the subgroup of $A_n$ generated by the images of all the degeneracies $s_i$ for $i\le j$. We show that the composite $N_jA_n\xrightarrow{pi}A_n/D_j(A_n)$ is an isomorphism for all $n$ and $1\le j\le n$. Let $\phi$ denote this composite.

We proceed by induction on $j$. Let $j=1$. Let $[x]$ be a class in $A_n/s_1(A_{n-1})$. A representative for $[x]$ is $x-s_1d_{1,1}x$ and $d_{1,1}(x-s_1d_{1,1}x)=0$. Therefore $\phi$ is onto. Suppose that $x$ is such that $d_{1,1}x=0$ and that $x=s_1y$, i.e, that $\phi(x)=0$. Then 
\[0=d_{1,1}x=d_{1,1}s_1y=y\] 
and therefore $x=0$. Thus $\phi$ is also injective.

Now suppose that the map $\phi:N_kA_n\to A_n/D_k(A_n)$ is known to be an isomorphism if $k< j$ and consider the map $\phi:N_jA_n\to A_n/D_j(A_n)$.
We have the following commutative diagram:
\begin{center}
\begin{tikzcd}
N_{j-1}A_n\arrow[r,"\phi" above,"\cong" below]&A_n/D_{j-1}(A_n)\arrow[d,two heads]\\
N_jA_n\arrow[r,"\phi"]\arrow[u,hook]&A_n/D_{j}(A_n)
\end{tikzcd}
\end{center}
From the diagram, we see that any class $[x]\in A_n/D_j(A_n)$ can be represented by an element $x\in N_{j-1}A_n$. However, $x-s_jd_{j,1}x$ is in $N_jA_n$ and represents $[x]$, so the bottom map $\phi$ in the diagram is onto. The cubical identities imply that the degeneracy $s_j:A_{n-1}\to A_n$ maps $N_{j-1}A_{n-1}$ into $N_{j-1}A_n$ and takes $D_{j-1}(A_{n-1})$ to $D_{j-1}(A_n)$, and so there is a commutative diagram
\begin{center}
\begin{tikzcd}
N_{j-1}A_{n-1}\arrow[d,"s_j"]\arrow[r,"\phi" above,"\cong" below]&A_{n-1}/D_{j-1}(A_{n-1})\arrow[d,"s_j"]\\
N_{j-1}A_n\arrow[r,"\phi"]&A_n/D_{j-1}(A_n)
\end{tikzcd}
\end{center}
Furthermore, the sequence 
\[0\to A_{n-1}/D_{j-1}(A_{n-1})\xrightarrow{s_j}A_n/D_{j-1}(A_n)\to A_n/D_j(A_n)\to 0\]
is exact. Therefore, if $\phi(x)=0$ for some $x\in N_jA_n$, then $x=s_jy$ for some $y\in N_{j-1}A_n$. However, $d_{j,1}x=0$ and therefore 
\[0=d_{j,1}x=d_{j,1}s_jy=y\]
so that $x=0$ and therefore $\phi$ is injective.
\end{proof}

\begin{example}
\label{example:cubical_abelian_group}
Let $X$ be a closure space. Then $\mathbb{Z}\mathscr{C}^{(J,\otimes)}(X)$ are the cubical abelian groups, where $\mathbb{Z}\mathscr{C}^{(J,\otimes)}(X)_n$ are the free abelian group on the sets $\mathscr{C}^{(J,\otimes)}(X)_n$, for $n\ge 0$. These are in fact a cubical abelian group with connections, which follows from \cref{example:cubical_group_with_connections}.
\end{example}

 Furthermore, there is an equivalence of categories $\mathbf{ccAb}$ and $\mathbf{Ch}_{\ge 0}(\mathbf{Ab})$, given by the normalization functor $N$.

\begin{theorem}{\cite[Theorem 14.8.1]{brown2004nonabelian}}[Cubical Dold-Kan correspondence]
\label{theorem:dold_kan_for_cubical_sets_with_connections}
Let $\mathbf{A}$ be an additive category with kernels. The following categories, defined internally to $\mathbf{A}$, are equivalent.
\begin{itemize}
\item The category of chain complexes.
\item The category of crossed complexes.
\item The category of cubical $\mathbf{A}$-objects with connections. 
\item The category of cubical $\omega$-groupoids with connections.
\end{itemize}
\end{theorem}

\begin{example}
\label{example:cubical_dold_kan}
Following \cref{example:cubical_abelian_group} and \cref{theorem:dold_kan_for_cubical_sets_with_connections} we can associate to $\mathbb{Z}\mathscr{C}^{(J,\otimes)}(X)$ a chain complex of abelian groups. We still call it the \emph{Moore complex}.
\end{example}

\begin{definition}
\label{def:simplicial_homology}
Let $X$ be a closure space. Consider the Moore complex of $\mathbb{Z}\mathscr{S}^{J}(X)$, which we will denote by $\gsimpchain{n}{X}$. More specifically, $\gsimpchain{n}{X}$ is the free abelian group generated by all continuous $\sigma:\gsimp{n}\to X$. For each $n\ge 1$, we define a map $\partial_n:\gsimpchain{n}{X}\to \gsimpchain{n-1}{X}$  in the following way. For a given $\sigma:\gsimp{n}\to X$ we define 
\[\partial_n\sigma:=\sum_{i}(-1)^i\sigma|_{d_i(\gsimp{n})}\]
and extending linearly. It can be checked that $\partial^2=0$, and thus we have a chain complex of free abelian groups, $(\gsimpchain{\bullet}{X},\partial_{\bullet})$ whose homology groups we will denote by $\gsimphom{n}{X}$. Note that by \cref{theorem:normalized_simplicial_complex,theorem:inclusion_of_normalized_simplicial_complex}
, we could have used the normalized complex or the quotient by degenerate simplices complex instead of the Moore complex to get an equivalent definitions of singular simplicial homologies for closure spaces. We can also augment the complex $\gsimpchain{\bullet}{X}$ by setting $\gsimpchain{-1}{X}=\mathbb{Z}$ and let the $0$-boundary be $\varepsilon(\sum n_i\sigma_i)=\sum n_i$. We then get reduced homology groups $\gsimphomred{\bullet}{X}$. Note that for all $n\ge 1$, $\gsimphom{n}{X}=\gsimphomred{n}{X}$ and $\gsimphom{0}{X}\cong \gsimphomred{0}{X} \oplus\mathbb{Z}$.
\end{definition}

\begin{definition}
\label{def:cubical_homology_theory}
Let $X$ be a closure space. Define $\gcubehom{\bullet}{X}$ to be the homology groups of the normalized complex of $\mathbb{Z}\mathscr{C}^{(J,\otimes)}(X)$. By \cref{theorem:normalized_cubical_complex}, we can equivalently define them in the following way. Let $\gcubechainwhole{\bullet}{X}$ denote the unnormalized complex of $\mathbb{Z}\mathscr{C}^{(J,\otimes)}(X,c)$. Given a continuous $\sigma:\gcube{n}\to X$, for $1\le i\le n$ define
\begin{gather}
A_i^n(\sigma)(a_1,\dots, a_{n-1}):=\sigma (a_1,\dots ,a_{i-1},0,a_i,\dots, a_{n-1})\\
B_i^n(\sigma)(a_1,\dots, a_{n-1}):=\sigma (a_1,\dots ,a_{i-1},1,a_i,\dots, a_{n-1})
\end{gather}
Say $\sigma$ is \emph{degenerate} if $A_i^n\sigma=B_i^n\sigma$ for some $i$. Let 

\[\gcubechaindeg{n}{X}:=\mathbb{Z}\langle \{\sigma: (\gcube{n},c_1)\to (X,c)\,|\, \sigma\text{ is continuous and degenerate}\}\rangle\]
 
Then, 
\[\gcubechain{n}{X}:=\gcubechainwhole{n}{X}/\gcubechaindeg{n}{X}\,,\]
where $\gcubechainwhole{n}{X}=\mathbb{Z}\langle\{\sigma:(\iindcube{n},c_1)\to (X,c)\,|\,\sigma\text{ is continuous}\}\rangle$. Elements of $\gcubechain{n}{X}$ are called \emph{$(J,\otimes)$ singular cubical n-chains} in $X$. The \emph{boundary} of $\gcubechain{n}{X}$ is defined by 
\[\partial\sigma:=\sum_{i=1}^n(-1)^i(A_i^n\sigma-B_i^n\sigma)\]
and extending linearly. We can also augment the complex $\gcubechainwhole{\bullet}{X}$ by setting $\gcubechainwhole{-1}{X}=\mathbb{Z}$ and let the $0$-boundary be $\varepsilon(\sum n_i\sigma_i)=\sum n_i$. Doing the same construction, by quotienting out the subcomplex of degenerate cubes, we get reduced homology groups $\gcubehomred{\bullet}{X}$. Note that for all $n\ge 1$, $\gcubehom{n}{X}=\gcubehomred{n}{X}$ and $\gcubehom{0}{X}\cong \gcubehomred{0}{X} \oplus\mathbb{Z}$.
\end{definition}

\section{Homotopy in closure spaces}

\label{section:homotopy}
In this section we recall various homotopy theories for closure spaces \cite{bubenik2021applied}. There is a forgetful functor from closure spaces to the category of sets, $U:\cat{Cl}\to \cat{Set}$, which forgets the closure structure. The category $\cat{Set}$ together with the Cartesian product of sets, $\times$ and the one element set, $*$, forms a symmetric monoidal category, $(\cat{Set},\times, *)$. Observe also, that $*$, with the unique closure operation one can put on it, is the terminal object in $\cat{Cl}$.

\subsection{Products and intervals}
We start by recalling facts about product operations and interval objects associated to product operations for closure spaces \cite[Section 5]{bubenik2021applied}.

\begin{definition}
\label{def:product_operation}
A \emph{product operation} 
is a functor $\otimes:\cat{Cl}\times \cat{Cl}\to \cat{Cl}$
that yields a symmetric monoidal category structure on $\cat{Cl}$, $(\cat{Cl},\otimes,*)$ and also commutes with the forgetful functor $U:\cat{Cl}\to \cat{Set}$, and the symmetric monoidal category structure on $\cat{Set}$. That is, $(X,c) \otimes (Y,d) = (X \times Y, c \otimes d)$ for some closure operator $c \otimes d$ on $X \times Y$, for all $(X,c)$, $(Y,d)$ in $\cat{Cl}$, and the associator, unitor and braiding morphisms are given by those on $\cat{Set}$. 
\end{definition}

It is possible to define a partial order structure among product operations on $\cat{Cl}$. For product operations, $\otimes_1$, $\otimes_2$, say $\otimes_1\le \otimes_2$ if there exists a natural transformation $\alpha:\otimes_1\to \otimes_2$ such that for all closure spaces $X$ and $Y$, $\alpha_{(X,Y)}:X\otimes_1 Y\to X\otimes_2 Y$ is the identity map.  

\begin{example}
\label{example:products}
The product closure and inductive product closures are examples of product operations, and furthermore $\boxdot\le\otimes \le \times$, for any product operation $\otimes$.
\end{example}

\begin{definition}
\label{def:interval_object}
Given a product operation $\otimes$, an \emph{interval } for $\otimes$
is a closure space $J$, with two maps $*\xrightarrow{0}J\xleftarrow{1}*$ such that the induced map  $0\coprod 1:*\coprod *\to J$ is one-to-one, and that there exists a symmetric, associative, morphism $\vee:J\otimes J\to J$ for which $0$ is its \emph{neutral} element and $1$ its \emph{absorbing} element. In other words, $0\vee t=t$ and $1\vee t=1$ for all $t\in J$.
A \emph{morphism of intervals} is a continuous map $f: J \to K$ such that $f(0_J)=0_K,f(1_J)=1_K$, and $f(s\vee_J t)=f(s)\vee_K f(t)$. When there exists a morphism of intervals $f:J\to K$, we write $J\le K$.
\end{definition}

\begin{lemma}{\cite[Lemma 5.9]{bubenik2021applied}}
\label{lemma:order_of_intervals}
Let $\otimes_1$ and $\otimes_2$ be product operations such that $\otimes_1\le \otimes_2$. If $J$ is an interval for $\otimes_2$, then $J$ is also an interval for $\otimes_1$.
\end{lemma}

The following are interval objects for any product operation $\otimes$ 
\begin{example}{\cite[Example 5.4]{bubenik2021applied}}
  \begin{enumerate}
  \item $I$ with the inclusions of the points $0$ and $1$.
  \item For $m \geq 1$, $J_m$ with the inclusions of the points $0$ and $m$ is an interval object.
  \item For $m \geq 1$ and $0 \leq k \leq 2^m-1$, $J_{m,k}$ with the inclusions of the points $0$ and $m$ is an interval object.
  \end{enumerate}
\end{example}

\subsection{Homotopy between maps}
Let $(X,c)$ and $(Y,d)$ be closure spaces. Here we define various equivalence relations $\sim_{(J,\otimes)}$ on the set $\cat{Cl}((X,c),(Y,d))$, for a given choice of product operation $\otimes$ and an interval $J$ for $\otimes$.

For a closure space $(X,c)$,  a product operation $\otimes$ and an interval $J$ for $\otimes$, we have the following natural map
\[X\coprod X\cong X\otimes (*\coprod *)\xrightarrow{\mathbf{1}_{X}\otimes(0\coprod 1)} X \otimes J\]

\begin{definition}
\label{def:homotopy}
Let $f,g:(X,c)\to (Y,d) \in \cat{Cl}$.
Let $\otimes$ be a product operation and let $J$ an interval for $\otimes$. Say that $f$ and $g$ are \emph{one-step $(J,\otimes)$-homotopic} if there exists a morphism map $H$ such that the following diagram commutes.
\begin{center}
\begin{tikzcd}
X\coprod X\arrow[r,"f\coprod g"]\arrow[d]  & Y \\
X\otimes J\arrow[ru,dashed,"H"']
\end{tikzcd}
\end{center}
We call $H$ a \emph{one-step $(J,\otimes)$ homotopy}.
Let $\sim_{(J,\otimes)}$ be the equivalence relation on the set $\cat{Cl}(X,Y)$ generated by one-step $(J,\otimes)$ homotopies.  
If $f\sim_{(J,\otimes)} g$, say $f$ and $g$ are \emph{$(J,\otimes)$-homotopic}.
\end{definition}

Let $J$ be a closure space with at least two elements. Equivalently, there exists a one-to-one map $0\coprod 1:*\coprod *\to J$ that identifies these two distinct elements. Given another such closure space $K$, we construct $J*K$ as a pushout diagram:

\begin{center}
\begin{tikzcd}
&J\arrow[dr,]\\
*\arrow[ru,"1_J"]\arrow[dr,swap,"0_K"]&&J*K\\
&K\arrow[ur]
\end{tikzcd}
\end{center}

We call $J*K$ the concatenation of $J$ with $K$.
Let $J^{*m}$ denote an $m$-fold concatenation of $J$ with itself, for some $m\ge 1$. If $J$ and $K$ are interval objects for $\otimes$, then so is $J*K$ (\cite[Lemma 5.15]{bubenik2021applied}).

\begin{lemma}{\cite[Lemma 5.29]{bubenik2021applied}}
\label{lemma:homotopy_and_concatenation}
Let $\otimes$ be a product operation and let $J$ be an interval for $\otimes$ and let $f,g:X\to Y$ be continuous maps of closure spaces. Then $f\sim_{(J,\otimes)}g$ iff $f$ and $g$ are one-step $(J^{*m},\otimes)$ homotopic for some $m\ge 1$. 
\end{lemma}

\begin{proposition}
\label{prop:order_between_homotopies}
Let $\otimes_1$ and $\otimes_2$ be product operations with $\otimes_1\le\otimes_2$ and let $J,K$ be intervals for $\otimes_2$ such that $J\le K$. Suppose that $f,g:X\to Y$ are one-step $(K,\otimes_2)$-homotopic. Then $f$ and $g$ are also one-step $(J,\otimes_1)$-homotopic.
\end{proposition}

\begin{proposition}{\cite[Theorem 5.38]{bubenik2021applied}}
\label{prop:homotopy_implications}
Let $X$ and $Y$ be closure spaces. Let $f,g:X\to Y$ be continuous maps. Then we have the following implications 
\begin{equation*}
\begin{tikzcd}
    f \sim_{(J_1,\times)} g \ar[r,Rightarrow] \ar[d,Rightarrow] & f \sim_{(J_+,\times)} g \ar[r,Rightarrow] \ar[d,Rightarrow] & f \sim_{(I,\times)} g \ar[d,Rightarrow]  \\
    f \sim_{(J_1,\boxdot)} g \ar[r,Rightarrow] & f \sim_{(J_+,\boxdot)} g \ar[r,Rightarrow] & f \sim_{(I,\boxdot)} g
  \end{tikzcd}
\end{equation*}
Moreover, the relation $\sim_{(J_1,\times)}$ is maximal. In other words, it implies any possible homotopy relation induced by any interval $J$ for any product operation $\otimes$. That is, the relation $\sim_{(J_1,\times)}$ is the finest homotopy relation we can have in $\cat{Cl}$.
\end{proposition}

\begin{definition}
\label{def:homotopy_equivalent_spaces}
We say two closure spaces $X$ and $Y$ are \emph{$(J,\otimes)$-homotopy equivalent} if there exist continuous $f:X\to Y$ and $g:Y\to X$ such that $fg\sim_{(J,\otimes)}\mathbf{1}_Y$ and $gf\sim_{(J,\otimes)}\mathbf{1}_X$. We say a closure space $X$ is \emph{$(J,\otimes)$-contractible} if it is $(J,\otimes)$-homotopy equivalent to the one point space.
\end{definition}

\subsection{Path-connectedness}
Here we define various notions paths and path-connectedness for closure spaces.

\begin{definition}
\label{def:path_equivalence_relation}
Let $X$ be a closure space. For $x,y\in X$, we say $x$ and $y$ are \emph{$J$-one-step path-connected} if there exists a continuous map $\alpha:J\to X$ such that $\alpha(0)=x$ and $\alpha(1)=y$ and we call the map $\alpha$ a \emph{$J$-one-step path}. Let $x\xrightarrow{J}y$ be the equivalence relation on $X$ generated by  $J$-one-step path-connectedness. If $x\xrightarrow{J}y$ we say that $x$ and $y$ are $J$-path connected. We call the equivalence classes \emph{$J$ path components}.
We label the set of equivalence classes of $J$ path components by $\pi_0^J(X)$.
\end{definition}

It follows immediately from the definition that we have the following characterization of $J$-path connectedness.

\begin{lemma}
\label{lemma:path_equivalence_relation}
Let $X$ be a closure space and let $x,y\in X$. Then $x\xrightarrow{J}y$ if and only if there exists a finite sequence of points $z_0,\dots, z_m\in X$, $x=z_0$, $y=z_m$ such that there exists a $J$-one-step path $\alpha_i:J\to X$ such that $\alpha_i(0)=z_i$ and $\alpha_i(1)=z_{i+1}$ or $\alpha_i(1)=z_i$ and $\alpha_i(0)=z_{i+1}$, for all $0\le i <m$.
\end{lemma}

\begin{definition}
\label{def:path_concatenation}
Let $\alpha,\beta:J\to X$ be $J$-paths such that $\alpha(1)=\beta(0)$. We define $\alpha*\beta$ to be the \emph{concatenation of $\alpha$ and $\beta$}. That is, $\alpha*\beta $ is the unique morphism in the pushout diagram
\begin{center}
\begin{tikzcd}
&J\arrow[drr,"\alpha"]\arrow[rd]\\
*\arrow[ru,"1"]\arrow[dr,swap,"0"]&&J*J\arrow[r,dashed,"\alpha*\beta"]& X\\
&J\arrow[urr,swap,"\beta"]\arrow[ru]
\end{tikzcd}
\end{center}
\end{definition}

\begin{lemma}
\label{lemma:path_equivalence_relation_2}
Let $X$ be a closure space and let $x,y\in X$. Then $x\xrightarrow{J}y$ if and only if there exists an $m\ge 1$ such that $x$ and $y$ are one-step $J^{*m}$ path-connected.
\end{lemma}

\begin{proof}
By \cref{lemma:path_equivalence_relation} we know that there is a finite sequence of points in $X$, $z_0,\dots, z_m\in X$, $x=z_0$, $y=z_m$ and $J$-one-step paths $\alpha_i:J\to X$ such that $\alpha_i(0)=z_i$ and $\alpha_i(1)=z_{i+1}$ or $\alpha_i(1)=z_i$ and $\alpha_i(0)=z_{i+1}$, for all $0\le i <m$. Then the concatenation $\alpha_0*\alpha_1\cdots *\alpha_m$ gives us a  $J^{*m}$-one-step path between $x$ and $y$. Going in the reverse direction, given a $J^{*m}$-one-step path between $x$ and $y$, we can split it into $J$-one step paths $\alpha_i$, that by \cref{lemma:path_equivalence_relation} tell us that $x\xrightarrow{J}y$.
\end{proof}

\begin{corollary}
\label{corollary:path_equivalence_relation}
Let $X$ be a closure space and let $x,y\in X$. 
\begin{enumerate}
\item \label{it:I}$x\xrightarrow{I}y$ iff $x$ and $y$ are $I$-one-step path-connected.
\item \label{it:J_1} $x\xrightarrow{J_1} y$ iff there exists $m \geq 1$ such that $x$ and $y$ are $J_m$-one-step path-connected.
\item $x\xrightarrow{J_+} y$ iff there exists $m \geq 1$ and $0 \leq k\leq 2^m-1$ such that $x$ and $y$ are $J_{m,k}$-one-step path-connected.
\end{enumerate}
\end{corollary}

\begin{proof}
Note that for \eqref{it:I}, given  $I^{*m}(=[0,m])$, we can always construct a reparametrization homeomorphism $h:I\to [0,m]$. Thus all the statements follow from \cref{lemma:path_equivalence_relation,lemma:generalized_intervals,lemma:path_equivalence_relation_2}.
\end{proof}

\begin{lemma}
\label{prop:implications_of_path_connectedness}
Let $J,K$ be closure spaces with distinguished points $0_J,1_J,0_K,1_k$ and suppose that there is a continuous map $f:J\to K$ such that $f(0_J)=0_K$ and $f(1_J)=1_K$. Let $X$ be a closure space and let $x,y\in X$. Then $x\xrightarrow{K}y$ implies $x\xrightarrow{J}y$.
\end{lemma}

\begin{proof}
Given a $K$-one-step path $\alpha:K\to X$, we can precompose it with a given morphism $f:J\to K$ to obtain a $J$-one-step path $\alpha f:J\to X$. Thus the claim follows.
\end{proof}

By \cref{prop:implications_of_path_connectedness,lemma:comparison_of_intervals}
 we have the following result.
 
\begin{corollary}
\label{corollary:implications_of_path_connectedness}
Let $X$ be a closure space. Then $x\xrightarrow{J_1}y\Longrightarrow x\xrightarrow{J_+}y\Longrightarrow x\xrightarrow{I}y$.
\end{corollary}

\begin{lemma}
\label{lemma:path_connectedness_of_image}
Let $f:X\to Y$ be a continuous map of closure spaces. If $X$ is $J$-path-connected, then so is $f(X)\subseteq Y$.
\end{lemma}

\begin{proof}
We can compose any $J$-one-step path $\alpha:J\to X$ with $f$, to get a $J$-one-step path $f\alpha:J\to Y$, whose image is contained in $f(X)$. Thus the result follows.
\end{proof}

From now on, let $J$ denote either the spaces $I$, $J_1$ or $J_+$.

\begin{proposition}
\label{prop:homology_decomposition_into_path_components}
Let $X$ be a closure space. Then $\gcubehom{n}{X}\cong \oplus_{\alpha}\gcubehom{n}{X_{\alpha}}$ and $\gsimphom{n}{X}\cong \oplus_{\alpha}\gsimphom{n}{X_{\alpha}}$, where $\{X_{\alpha}\}_{\alpha\in A}$ are the $J$-path components of $X$.
\end{proposition}

\begin{proof}
Note that each $\gcube{n}$ and $\gsimp{n}$ is $J$-path-connected. Thus each $(J,\otimes)$ singular $n$-cube and each $(J,\times)$ singular $n$-simplex has a $J$-path-connected image, by \cref{lemma:path_connectedness_of_image}. Thus, each $\gcubechain{n}{X}$ and $\gsimpchain{n}{X}$ split into direct sum of their subgroups $\gcubechain{n}{X_{\alpha}}$ and $\gsimpchain{n}{X_{\alpha}}$, respectively. The boundary maps of the two chain complexes will respect this decomposition, hence the corresponding homology groups also split into direct sums on the $J$-path components accordingly. 
\end{proof}

\begin{proposition}
\label{prop:0_homology_counts_path_components}
Let $X$ be a closure space. Then $\gcubehom{0}{X}\cong \gsimphom{0}{X}\cong \oplus_{a\in \pi_0^J(X)}\mathbb{Z}$.
\end{proposition}

\begin{proof}
By definition, $\gsimphom{0}{X}=\gsimpchain{0}{X}/\text{Im}\partial_1$, since $\partial_0=0$. Define a homomorphism $\varepsilon:\gsimpchain{0}{X}\to \mathbb{Z}$ by $\varepsilon(\sum_in_i\sigma_i)=\sum_i n_i$. If $X$ is nonempty, $\varepsilon$ is clearly surjective. Suppose that $X$ is $J$-path-connected. Observe that $\text{Im}\partial_1\subseteq \ker \varepsilon$, since for a $\sigma:\gsimp{1}\to X$ we have $\varphi\partial_1(\sigma)=\varepsilon(\sigma(1)-\sigma(0))=1-1=0$. Now let $\sum_in_i\sigma_i$ be in $\gsimpchain{0}{X}$ and suppose that $\varepsilon(\sum_in_i\sigma_i)=0$. The $\sigma_i$'s can be identified with points of $X$. Let $x_i$ be the image of $\sigma_i$ in $X$. Since $X$ is assumed to be $J$-path-connected we have that $x_0\xrightarrow{J}x_i$ for all $i$. By \cref{lemma:path_equivalence_relation} this means that for each $i$, there is a finite sequence of points in $X$, $x_0=z_{i,0},z_{i,1},\dots, z_{i,m}=x_i$ such that there exists a $J$-one-step path $\tau_{i,j}:J\to X$ such that $\tau_{i,j}(0)=z_{i,j}$ and $\tau_{i,j}(1)=z_{i,j+1}$ or $\tau_{i,j}(1)=z_{i,j}$ and $\tau_{i,j}(0)=z_{i,j+1}$, for all $i$, $0\le j <m$. Technically, the number of points between $x_0$ and $x_i$ could be different for different $i$'s, but we can always add constant points to the list to make the number uniform, say some $m$. Note that each such $J$ one-step path is then a $(J,\times)$ singular $1$ simplex $\tau_{i,j}:J\to X$. For each $i$, Let $\tau_i$ be the $(J,\times)$ singular $1$ chain $\tau_i:=\sum_{j=0}^m\lambda_{i,j}\tau_{i,j}$ where $\lambda_{i,j}=1$ if it is the case that $\tau_{i,j}(0)=z_{i,j}$ and $\tau_{i,j}(1)=z_{i,j+1}$ and $\lambda_{i,j}=-1$ in the other case. Note that $\partial \tau_i=\partial \sum_j\lambda_{i,j}\tau_{i,j}$ equals $x_0-x_i$ or $x_i-x_0$, since $\tau_{i}$ was defined with the purpose of being telescoping after the application of $\partial$. The case in which we end up in depends on the constants $\lambda_{i,j}$. Without loss of generality, we can change the sign of $\tau_i$ so that $\partial \tau_i=x_i-x_0$ for all $i$.
Then $\partial(\sum_i n_i\tau_i)=\sum_in_i\sigma_i-\sum_{i}n_i\sigma_0=\sum_in_i\sigma_i$ since $\varepsilon(\sum_i n_i\sigma_i)=\sum_in_i=0$ by assumption. Thus $\sum_in_i\sigma_i$ is a boundary, which shows that $\ker \varepsilon\subseteq \text{Im}\partial_1$. Hence $\varepsilon$ induces an isomorphism $\gsimphom{0}{X}\cong \mathbb{Z}$. The general statement then follows by \cref{prop:homology_decomposition_into_path_components}.
\end{proof}

\subsection{Homotopy Invariance of Homology}
Let $J$ be either $I$, $J_1$ or $J_+$. Let $\otimes$ be a product operation.
Here we show that $(J,\otimes)$-homotopic maps induce the same maps on homology groups $\gcubehom{n}{X}$ and $\gsimphom{n}{X}$. 

\begin{definition}
\label{def:induced_maps_on_cubical_homology}
Let $f:X\to Y$ be a continuous map of closure spaces. For every $n\ge 0$, we define a group homomorphism $f_{\#}:\gcubechainwhole{n}{X}\to \gcubechainwhole{n}{Y}$ by $f_{\#}(\sigma)=f\circ \sigma$. If $\sigma$ is degenerate, then so is $f\circ \sigma$ and thus $f_{\#}(\gcubechaindeg{n}{X})\subseteq \gcubechaindeg{n}{Y}$. Therefore, $f_{\#}$ induces a group homomorphism $f_{\#}:\gcubechain{n}{X}\to \gcubechain{n}{Y}$ for every $n\ge 0$. Moreover, it is straightforward to show that $f_{\#}$ commutes with the boundary operators of $\gcubechain{\bullet}{X}$ and $\gcubechain{\bullet}{Y}$. Therefore, for all $n\ge 0$, $f_{\#}$ induces a homomorphism of quotient groups $f_*:\gcubehom{n}{X}\to \gcubehom{n}{Y}$. Furthermore, the following diagram also commutes
\begin{center}
\begin{tikzcd}
\gcubechain{0}{X}\arrow[dr,"\varepsilon"]\arrow[rr,"f_{\#}"]&&\gcubechain{0}{Y}\arrow[dl,swap,"\varepsilon"]\\
&\mathbb{Z}
\end{tikzcd}
\end{center}
Hence $f_{\#}$ also induces a homomorphisms $f_*:\gcubehomred{0}{X}\to \gcubehomred{0}{Y}$.
\end{definition}

\begin{definition}
\label{def:induced_maps_on_simplicial_homology}
Let $f:X\to Y$ be a continuous map of closure spaces. For every $n\ge 0$, we define a group homomorphism $f_{\#}:\gsimpchain{n}{X}\to \gsimpchain{n}{Y}$ by $f_{\#}(\sigma)=f\circ \sigma$. Moreover, it is straightforward to show that $f_{\#}$ commutes with the boundary operators of $\gsimpchain{\bullet}{X}$ and $\gsimpchain{\bullet}{Y}$. Therefore, for all $n\ge 0$, $f_{\#}$ induces a homomorphism of quotient groups $f_*:\gsimphom{n}{X}\to \gsimphom{n}{Y}$. Furthermore, the following diagram also commutes
\begin{center}
\begin{tikzcd}
\gsimpchain{0}{X}\arrow[dr,"\varepsilon"]\arrow[rr,"f_{\#}"]&&\gsimpchain{0}{Y}\arrow[dl,swap,"\varepsilon"]\\
&\mathbb{Z}
\end{tikzcd}
\end{center}
Hence $f_{\#}$ also induces a homomorphisms $f_*:\gsimphomred{0}{X}\to \gsimphomred{0}{Y}$.
\end{definition}

\begin{theorem}{\cite[Theorem 6.6]{bubenik2021applied}}[\textbf{Part 1 of  Theorem 1.1}]
\label{theorem:cubical_homotopy_invariance}
Let $f,g:X\to Y$ be two continuous maps of closure spaces and suppose that  $f\sim_{(J,\otimes)}g$. Then $f_*=g_*:\gcubehom{\bullet}{X}\to \gcubehom{\bullet}{Y}$.
\end{theorem}

\begin{theorem}[\textbf{Part 2 of Theorem 1.1}]
\label{theorem:simplicial_homotopy_invariance}
Let $f,g:X\to Y$ be maps of closure spaces and suppose that $f\sim_{(J,\times)}g$. Then $f_*=g_*:\gsimphom{n}{X}\to \gsimphom{n}{Y}$ for all $n\ge 0$.  
\end{theorem}

\begin{proof}
 Let $m$ be such that $H:X\times J^{*m}\to Y$ is a one-step $(J^{*m},\times)$ homotopy between $f$ and $g$. Such an $m$ exists by \cref{lemma:homotopy_and_concatenation}. Consider the space $\gsimp{n}\times J^{*m}$. We subdivide $\indsimp{n}\times J^{*m}$ in the ``correct" manner. Let us use the notation $\gsimp{n}=[v_0,v_1, \dots, v_n]$. In the case $J=I$, what we mean by this is the convex hull of the vertices $v_0,\dots, v_n$ that have the coordinates of the standard $n$-simplex, in $\mathbb{R}^{n+1}$. In the other cases, the underlying set is $\{0,\dots, n\}$ with the indiscrete topology, in the case $J=J_1$ or the topology with the opens being the down-sets, in the case $J=J_+$.

 We can pass from $\gsimp{n}\times \{i\}$ to $\gsimp{n}\times \{i+1\}$, in $\gsimp{n}\times J^{*m}$, for all $0\le i\le m-1$ by interpolating a  sequence of $(J,\times)$ $n$-simplices, each obtained from the preceding one by moving one point $(v_j,i)$ to $(v_j,i+1)$, starting with $(v_n,i)$ and working backwards to $(v_0,i)$. For example, the first step is moving $[(v_0,i),\dots ,(v_n,i)]$ to $[(v_0,i),\dots ,(v_{n-1},i),(v_n,i+1)]$, the second step is to move $[(v_0,i),\dots ,(v_{n-1},i),(v_n,i+1)]$ to $[(v_0,i),\dots ,(v_{n-1},i),(v_{n-1},i+1),(v_n,i+1)]$ etc... In the typical step $[(v_0,i),\dots ,(v_j,i),(v_{j+1},i+1),\dots ,(v_{n-1},i),(v_n,i+1)]$ moves up to $[(v_0,i),\dots ,(v_{j-1},i),(v_j,i+1,\dots ,(v_n,i+1))]$. Note that $[(v_0,i),\dots (v_j,i),(v_j,i+1),\dots ,(v_n,i+1)]$ is a $(J,\times)$ $(n+1)$-simplex that has $[(v_0,i),\dots (v_j,i), (v_{j+1},1),\dots ,(v_n,i+1)]$ as its ``lower" face and $[(v_0,i),\dots ,(v_{j-1},i),(v_j,i+1),\dots ,(v_n,i+1)]$ as its ``upper" face. Finally, observe that $\gsimp{n}\times J^{*m}$ with the $\ell_{\infty}$ metric is the union of such is the union of such $(J,\times)$ $(n+1)$-simplices, each intersecting the next in an $(J,\times)$ $n$-simplex face. 
 
 Now let $\sigma:\gsimp{n}\to (X,c_X)$ be a $(J,\times)$ singular $n$-simplex. Consider the composition $H\circ (\sigma\times \mathbf{1}):\gsimp{n}\times J^{*m}\to X\times J^{*m}\to Y$. We can then define $P:\gsimpchain{n}{X}\to \gsimpchain{n+1}{Y}$ by:
\[P(\sigma)=\sum_{i=0}^{m-1} \sum_{j=0}^n(-1)^jH\circ(\sigma\times \mathbf{1})|_{[(v_0,i),\dots, (v_j,i),(v_j,i+1),\dots , (v_n,i+1)]}\]
and extending linearly. We now show that $\partial P=g_{\#}-f_{\#}-P\partial$, i.e. $P$ is a chain homotopy between $f_{\#}$ and $g_{\#}$. We have
\begin{gather*}\partial P(\sigma)=\sum_{i=0}^{m-1}\Big(\sum_{k\le j}(-1)^j(-1)^k H\circ(\sigma\times \mathbf{1})|_{[(v_0,i),\dots ,(v_{k-1},i),\widehat{(v_k,i)},(v_{k+1},i),\dots, (v_j,i),(v_j,i+1) ,\dots ,(v_n,i+1)]}\\
+\sum_{k\ge j}(-1)^j(-1)^{k+1}H\circ(\sigma\times \mathbf{1})|_{[(v_0,i),\dots ,(v_j,i),(v_j,i+1),\dots ,(v_{k-1},i+1),\widehat{(v_k,i+1)},(v_{k+1},i+1),\dots ,(v_n,i+1)]}\Big)
\end{gather*}
The terms with $j=k$ in the inner sum always cancel for each $0\le i\le m-1$, except for $H\circ(\sigma\times \mathbf{1})|_{[\widehat{(v_0,i)},(v_0,i+1),(v_1,i+1),\dots ,(v_n,i+1)]}$, which is $H(-,i+1)\circ \sigma=H(-,i+1)_{\#}(\sigma)$, and $-H\circ (\sigma\times \mathbf{1})|_{[(v_0,i),(v_1,i),\dots ,(v_n,i),\widehat{(v_n,i+1)}]}$, which is $-H(-,i)\circ \sigma=-H(-,i)_{\#}(\sigma)$. Thus, as we iterate through $i$, we get a telescoping series where the only surviving terms in the sum will be $H(-,m)_{\#}(\sigma)=g_{\#}(\sigma)$ and $-H(-,0)_{\#}(\sigma)=-f_{\#}(\sigma)$. The rest of the terms sum up to $-P\partial(\sigma)$ since
\begin{gather*}
P\partial(\sigma)=\sum_{i=0}^{m-1}\Big(\sum_{j<k}(-1)^j(-1)^kH\circ (\sigma\times \mathbf{1})|_{[(v_0,i),\dots, (v_j,i),(v_j,i+1),\dots , (v_{k-1},i+1),\widehat{(v_k,i+1)},(v_{k+1},i+1),\dots ,(v_n,i+1)]}\\
+\sum_{j>k}(-1)^{j-1}(-1)^kH\circ(\sigma\times \mathbf{1})|_{[(v_0,i),\dots ,(v_{k-1},i),\widehat{(v_k,i)},(v_{k+1},i),\dots ,(v_j,i),(v_j,i+1),\dots ,(v_n,i+1)]}\Big)
\end{gather*}
Thus $\partial P(\sigma)=g_{\#}(\sigma)-f_{\#}(\sigma)-P\partial(\sigma)$. Extending linearly we get $\partial P=g_{\#}-f_{\#}-P\partial$. Thus, $P$ is a chain homotopy between $f_{\#}$ and $g_{\#}$ and therefore $f_*=g_*$.
\end{proof}

\begin{example}
\label{example:cubes_are_contractible}
The space $\gcube{n}$ is $(J,\otimes)$-contractible. This follows from \cite[Example 4.53]{bubenik2021applied}.
\end{example}

\begin{example}
\label{example:simplices_are_contractible}
The space $\gsimp{n}$ is $(J,\times)$-contractible. The case $J=I$ is well-known. In the case, $J=J_1$, the space $\gsimp{n}$ is indiscrete as observed before and thus $(J_1,\times)$-contractible. In the case $J=J_+$, note that $\disimp{n}$ is an interval object for $\times$ and is thus $(\disimp{n},\times)$-contractible \cite[Lemma 5.21]{bubenik2021applied}. By \cite[Theorem 5.36]{bubenik2021applied}, this implies that $\disimp{n}$ is $(J_+,\times)$-contractible.
\end{example}

\section{Homology of Closure Spaces}

\label{section:homology}

Let $J\in \{J_+,J_1,I\}$ and $\otimes \in\{\times, \boxdot\}$ throughout this section. Recall the definition of a representable functor (\cref{def:representable_functor}).

\begin{definition}
\label{def:closure_space_models}
In the category of closure spaces, $\cat{Cl}$, we consider the models (\cref{section:acyclic_models})
 $\mathscr{M}^{(J,\otimes)}$, the collection of closure spaces that are $(J,\otimes)$-contractible.
\end{definition}

\begin{lemma}
\label{lemma:models_are_contractible}
For any model $M$ in $\mathscr{M}^{(J,\otimes)}$, we have $\gcubehomred{n}{M}=0$ and for any $M$ in $\mathscr{M}^{(J,\times)}$, $\gsimphomred{n}{M}=0$, for any $n\ge 0$. 
\end{lemma}

\begin{proof}
This follows from \cref{theorem:cubical_homotopy_invariance,theorem:simplicial_homotopy_invariance,example:homology_of_one_point_space}.
\end{proof}

\begin{example}
\label{example:homology_of_one_point_space}
Let $*$ be the one point space. Then $\gcubehomred{n}{*}=\gsimphomred{n}{*}=0$ for all $n\ge 0$.
\end{example}

\begin{lemma}
\label{lemma:homology_functors_are_representable}
The functors $\gcubechainwhole{n}{-},\gcubechain{n}{-},\gsimpchain{n}{-}$ and $\gsimpchainred{n}{-}$ are representable for all $n$.
\end{lemma}

\begin{proof}
The identity map $e_n:\gsimp{n}\to \gsimp{n}$ is a $(J,\times)$ singular $n$-simplex, and thus it is a class in $\gsimpchain{n}{\gsimp{n}}$. If $\sigma:\gsimp{n}\to X$ is a $(J,\times)$ singular $n$-simplex, the correspondence $\sigma\mapsto (\sigma,e_n)$ yields a representative of the functor $\gsimpchain{n}{-}$ for $n\ge 0$. For $n=-1$ the proof is trivial.

Let $\xi:\gsimpchain{\bullet}{}\to \gsimpchainnorm{\bullet}{}$ be the natural quotient map. The expression $\eta\sigma=(Id-d_0s_1)(Id-d_1s_2)\cdots (Id-d_{n-1}s_n)\sigma$ satisfies $\eta\sigma=0$ for all degenerate $(J,\times)$ singular $n-1$-simplices $\sigma$. Therefore, $\eta$ yields a map $\eta:\gsimpchainnorm{n}{}\to \gsimpchain{n}{}$. Furthermore, the composition $\xi\eta$ is the identity. The representability of $\gsimpchainnorm{\bullet}{}$ now follows by \cref{lemma:representable_functors}.

For the cubical case, the same arguments as above apply when we replace  $\gsimp{n}$ by $\gcube{n}$.
\end{proof}

\begin{theorem}
\label{theorem:simplicial_and_cubical_isomorphism}[\textbf{Theorem 1.2}]
There exist maps $f:\gsimpchainred{\bullet}{-}\to \gcubesimpchainred{\bullet}{-}$, $g:\gcubesimpchainred{\bullet}{-}\to \gsimpchainred{\bullet}{-}$ and chain homotopies $H:gf\simeq \text{identity}$, $G:fg\simeq \text{identity}$ such that $f$ and $g$ are identity in dimension $<2$ while $H$ and $G$ are $0$ in dimension $<2$.
\end{theorem}

\begin{proof}
We define a map $f:\gsimpchainred{\bullet}{-}\to \gcubesimpchainred{\bullet}{-}$ which is the identity in dimension $<2$. For each $(J,\times)$ singular $n$-simplex $\sigma$, let $f\sigma$ be the $(J,\times)$ singular $n$-cube defined by 
\[(f\sigma)(a_1,\dots ,a_n):=\sigma(v_0,\dots ,v_n),\]
where
\begin{gather*}
v_0=1-a_1, v_1=a_1(1-a_2),\dots ,\\
v_i=a_1\cdots a_i(1-a_{i+1}), 0<i<n,\dots ,\\
v_n=a_1\cdots a_n
\end{gather*}
It can be checked that $fd_i=B_{i+1}f$, $0\le i \le n$ and $fd_n=A_nf, A_if=s_nA_iA_nf$, $fd_n=s_{n+1}f$. These formulas imply that $f$ induces a map $f:\gsimpchainred{\bullet}{-}\to \gcubesimpchainred{\bullet}{-}$. Furthermore, $f\sigma=f\sigma'$ implies $\sigma=\sigma'$, and thus $f$ is an isomorphism of chain complexes. The statement of the theorem then follows from \cref{lemma:homology_functors_are_representable,example:homology_of_one_point_space,theorem:cubical_homotopy_invariance,theorem:simplicial_homotopy_invariance,theorem:representable_functor}.
\end{proof}

\subsection{Excision and Mayer-Vietoris exact sequence}
Let $X$ be a closure space and let $\mathcal{U}$ be an interior cover of $X$. Let $\gsimpchain{\bullet}{X}^{\mathcal{U}}\subseteq \gsimpchain{\bullet}{X}$, and $\gcubechain{\bullet}{X}^{\mathcal{U}}\subseteq \gcubechain{\bullet}{X}$ be the subchain complexes generateded by continuous $\sigma:\gsimp{n} \to X$ and $\tau:\gcube{n}\to X$, respectively, whose image is contained in some $U\in \mathcal{U}$. 

\begin{proposition}
\label{prop:chain_of_interior_cover}
$\gsimpchain{\bullet}{X}^{\mathcal{U}}$ is chain homotopy equivalent to $\gsimpchain{\bullet}{X}$ and $\gcubesimpchain{\bullet}{X}^{\mathcal{U}}$ is chain homotopy equivalent to $\gcubesimpchain{\bullet}{X}$. Furthermore, $\gsimpchain{\bullet}{X}^{\mathcal{U}}=\gsimpchain{\bullet}{X}$ and $\gcubesimpchain{\bullet}{X}^{\mathcal{U}}=\gcubesimpchain{\bullet}{X}$ when $J\in \{J_1,J_+\}$.
\end{proposition}

\begin{proof}
The case $\topchain{\bullet}{X}$ is covered in \cite{peterbubeniknikolamilicevic2022}. The proof is an adaptation of the classical barrycentric subdivision argument, see for example \cite{MR1867354}. The case $\topcubechain{\bullet}{X}$ then follows as well. Recall that we have a chain homotopy equivalence between $\topcubechain{\bullet}{X}$ and $\topchain{\bullet}{X}$ by \cref{theorem:simplicial_and_cubical_isomorphism}. Restricting this chain homotopy, gives us a chain homotopy between $\topcubechain{\bullet}{X}^{\mathcal{U}}$ and $\topchain{\bullet}{X}^{\mathcal{U}}$. Composing these chain homotopies gives a chain homotopy between $\topcubechain{\bullet}{X}^{\mathcal{U}}$ and $\topcubechain{\bullet}{X}$.

Let $J=J_+$. We only need to show that $\disimpchain{\bullet}{X}\subseteq \disimpchain{\bullet}{X}^{\mathcal{U}}$. Fix an $n\in \mathbb{N}$ and consider a continuous $\sigma:\disimp{n}\to X$. Note that $\disimp{n}$ is a compact topological space (as it is finite). By  \cref{lemma:finite_cover_whose_image_is_contained_in_covering_system}
it follows that there exists a finite open cover $\mathcal{V}$ of $\disimp{n}$ such that for every $V\in \mathcal{V}$ there exists a $U\in \mathcal{U}$ such that $\sigma(V)\subseteq U$. The only open sets of $\disimp{n}$ are the ordinals $[m]$ for $m\le n$. Thus, for it to be an open cover it must be the case that $[n]\in \mathcal{V}$. Thus we have that there exists a $U\in \mathcal{U}$ such that $\sigma([n])\subseteq U$ and thus $\disimpchain{\bullet}{X}\subseteq \disimpchain{\bullet}{X}^{\mathcal{U}}$. The case $\dicubechain{\bullet}{X}$ is similar, noting that $\dicube{n}$ is also a compact topological space and that every open cover of it has to have the whole set $\{0,1\}^n$ in it.

Let $J=J_1$. We only need to show that $\indsimpchain{\bullet}{X}\subseteq \indsimpchain{\bullet}{X}^{\mathcal{U}}$. Fix an $n\in \mathbb{N}$ and consider a continuous $\sigma:\indsimp{n}\to X$. Note that $(\indsimp{n},c_1)$ is a compact  topological space. Indeed it is the set with $n+1$ points equipped with the indiscrete topology. By \cref{lemma:finite_cover_whose_image_is_contained_in_covering_system}
there is an open cover $\mathcal{V}$ of $\indsimp{n}$ such that for every $V\in \mathcal{V}$ there is a $U\in \mathcal{U}$ so that $\sigma(V)\subseteq U$. Since $\indsimp{n}$ is given the indiscrete topology, the only open sets are $\emptyset$ and $\indsimp{n}$. Thus any open cover of $\indsimp{n}$ contains set $\indsimp{n}$. Thus $\sigma(\indsimp{n})\subseteq U$ for some $U\in \mathcal{U}$. Therefore $\indsimpchain{\bullet}{X}\subseteq \indsimpchain{\bullet}{X}^{\mathcal{U}}$. The case $\indcubechain{\bullet}{X}$ is similar, noting that $\indcube{n}$ is also a compact topological space and that the only open sets of it are $\emptyset$ and $\{0,1\}^n$.
\end{proof}

\begin{example}
\label{example:inductive_cube_not_compact}
The space $I^{\boxdot n}$ is not compact. For $x\in I$, let $B_{\frac{1}{4}}(x)$ be the open ball around $x$ of radius $\frac{1}{4}$. Then, for each $(x_1,x_2,\dots ,x_n)$ in $I^{\boxdot n}$, the collection of sets $\cup_{i=1}^n \{x_1\}\times \{x_2\}\times \dots \times B_{\frac{1}{4}}(x_i)\times \{x_{i+1}\}\times \dots \times \{x_n\}$ is an interior cover of $I^{\boxdot n}$ by \cref{def:inductive_product_closure}. However, it is clear that no finite subcollection of these sets is an interior cover (or even a cover) of $I^{\boxdot n}$.
\end{example}

\begin{remark}
\label{remark:excision_failure}
A natural question to ask is if we have chain homotopies between $\gcubechain{\bullet}{X}^{\mathcal{U}}$ and $\gcubechain{\bullet}{X}$ for $\otimes=\boxdot$. For the case $J=I$, one might hope that a method of barrycentric subdivision would be a proof tactic, as in for example the proof of Theorem 6.4 in \cite[Chapter 7]{MR1095046}, which proves the statement for $\otimes=\times$ and $J=I$ in the subcategory $\cat{Top}$. However, a crucial ingredient in that proof was the fact that $I^n$ was compact and we know $I^{\boxdot n}$ is not by \cref{example:inductive_cube_not_compact}. For the cases $J=J_1, J=J_+$ one might hope we could use a similar approach as in the proof of \cref{prop:chain_of_interior_cover}. However, \cref{example:excision_failure} shatters this hope. Thus, this question remains open.
\end{remark}

\begin{example}
\label{example:excision_failure}
We saw in the proof of \cref{prop:chain_of_interior_cover} that we had equalities $\disimpchain{\bullet}{X}^{\mathcal{U}}=\disimpchain{\bullet}{X}$ and $\indsimpchain{\bullet}{X}^{\mathcal{U}}=\indsimpchain{\bullet}{X}$. If we however replace $\times$ by $\boxdot$ it is however not true that we will have equalities $\idicubechain{\bullet}{X}^{\mathcal{U}}=\idicubechain{\bullet}{X}$ and $\iindcubechain{\bullet}{X}^{\mathcal{U}}=\iindcubechain{\bullet}{X}$. 

Let $X=\idicube{2}$. Then, it is straightforward to check that $A=\{(0,0)\}, B=\{(0,0),(1,0)\}$, $C=\{(0,0),(0,1)\}$ and $D=\{(0,1),(1,1),(1,0)\}$ is an interior cover of $\idicube{2}$. However, the image of the identity map $\idicube{2}\to \idicube{2}$ is clearly not contained in a single element of this cover. Thus $\idicubechain{\bullet}{X}^{\mathcal{U}}\neq\idicubechain{\bullet}{X}$.

Let $X=\iindcube{2}$. Then, it is straightforward to check that $A=\{(0,1),(0,0),(1,0)\}, B=\{(0,0),(1,0),(1,1)\}, C=\{(1,0),(1,1),(0,1)\}$ and $D=\{(0,0),(0,1),(1,1)\}$ is an interior cover of $\iindcube{2}$. However, the image of the identity map $\iindcube{2}\to \iindcube{2}$ is clearly not contained in a single element of this cover. Thus $\iindcubechain{\bullet}{X}^{\mathcal{U}}\neq\iindcubechain{\bullet}{X}$.

\end{example}

\begin{theorem}[\textbf{Theorem 1.5} (Excision)]
\label{theorem:excision}
Let $(X,c)$ be a closure space. Let $Z\subseteq A\subseteq X$ be such that $c(Z)\subseteq i(A)$. Then, the inclusion $(X-Z,A-Z)\to (X,A)$ induces isomorphisms $\gsimphom{n}{X-Z,X-A)}\cong \gsimphom{n}{X,A}$ and $\gcubesimphom{n}{X-Z,X-A}\cong \gcubesimphom{n}{X,A}$ for all $n$. Equivalently, for an interior cover $\{A,B\}$ of $(X,c)$, the inclusion $(B,A\cap B)\to (X,A)$ induces isomorphisms $\gsimphom{n}{B,A\cap B}\cong \gsimphom{n}{X,A}$ and $\gcubesimphom{n}{B,A\cap B}\cong \gcubesimphom{n}{X,A}$ for all $n$.
\end{theorem}

\begin{proof}
The equivalence of the two statements follows from setting $B=X-Z$ and $Z=X-B$. Then $A\cap B=A-Z$ and the condition $c(Z)\subseteq \text{int}_c(A)$ is equivalent to $X=\text{int}_c(A)\cup \text{int}_c(B)$.

We prove the second statement, involving an interior cover of $(X,c)$, $\mathcal{U}=\{A,B\}$. By \cref{prop:chain_of_interior_cover} we have that $\gsimpchain{\bullet}{X}^{\mathcal{U}}$ is chain homotopic to $\gsimpchain{\bullet}{X}$. Thus $\gsimpchain{\bullet}{X}^{\mathcal{U}}/\gsimpchain{\bullet}{A}$ is chain homotopic to $\gsimpchain{\bullet}{X}/\gsimpchain{\bullet}{A}$. The chain map $\gsimpchain{\bullet}{B}/\gsimpchain{\bullet}{A\cap B}\to \gsimpchain{\bullet}{X}^{\mathcal{U}}/\gsimpchain{\bullet}{A}$ induced by inclusion is an isomorphism since both quotients are free with basis the singular simplices in $B$ that do not lie in $A$. Thus we have an isomorphism $\gsimphom{n}{B,A\cap B}\cong \gsimphom{n}{X,A}$ for all $n$. The cubical case is similar.
\end{proof}

\begin{theorem}[\textbf{Theorem 1.4} (Mayer-Vietoris long exact sequence)]
\label{theorem:mayer_vietoris}
Let $X$ be a closure space and let $\{A,B\}$ be an interior cover of $X$. Then we have the following long exact sequences:
\begin{gather*}\cdots\to \gsimphom{n}{A\cap B}\to \gsimphom{n}{A}\oplus \gsimphom{n}{B}\to \gsimphom{n}{X}\to \gsimphom{n-1}{A}\to\cdots\\
\cdots\to \gcubesimphom{n}{A\cap B}\to \gcubesimphom{n}{A}\oplus \gcubesimphom{n}{B}\to \gcubesimphom{n}{X}\to \gcubesimphom{n-1}{A}\to\cdots
\end{gather*}
\end{theorem}

\begin{proof}
By \cref{prop:chain_of_interior_cover} we have a chain homotopy between chain complexes $\gsimpchain{\bullet}{X}^{\{A,B\}}$ and $\gsimpchain{\bullet}{X}$. We also have the following short exact sequences of chain complexes
\begin{gather*}
\gsimpchain{\bullet}{A\cap B}\xrightarrow{\varphi} \gsimpchain{\bullet}{A}\oplus \gsimpchain{\bullet}{B}\xrightarrow{\psi} \gsimpchain{\bullet}{X}^{\{A,B\}}\\
\gcubesimpchain{\bullet}{A\cap B}\xrightarrow{\varphi} \gcubesimpchain{\bullet}{A}\oplus \gcubesimpchain{\bullet}{B}\xrightarrow{\psi} \gcubesimpchain{\bullet}{X}^{\{A,B\}}
\end{gather*}
where $\varphi(x)=(x,-x)$ and $\psi(x,y)=x+y$. By standard arguments in homological algebra, a short exact sequence of chain complexes induces a long exact sequence in homology. By \cref{prop:chain_of_interior_cover}, it follows that the long exact sequences induced in homology are precisely the ones in the statement of the theorem.
\end{proof}

By \cref{remark:excision_failure} it is still an open question whether or not we have analogues of \cref{theorem:mayer_vietoris,theorem:excision} in the case of $\otimes=\boxdot$.

\subsection{Relative long exact sequences}
Here we show the existence of long exact sequences in homology associated to  pair of closure space $(X,A)$.

\begin{definition}
\label{def:good_pairs}
Let $X$ be a closure space and let $A\subseteq X$ be a subspace. We say that $(X,A)$ is a $(J,\otimes)$-\emph{good pair} if there exists a neighborhood $B$ of $A$ that $(J,\otimes)$ deformation retracts onto $A$.
\end{definition}

\begin{definition}
\label{def:relative_simplicia_homology}
Let $X$ be a closure space and let $A$ be a subspace. Then note that $\gsimpchain{n}{A}\subseteq \gsimpchain{n}{X}$. Let $ \gsimpchain{n}{X,A}:=\gsimpchain{n}{X}/\gsimpchain{n}{A}$. Note that the boundary map $\partial_n:\gsimpchain{n}{X}\to \gsimpchain{n-1}{X}$ descends to this quotient. We thus get a new chain complex of abelian groups $(\gsimpchain{\bullet}{X,A},\partial_{\bullet})$. Denote the homology groups of this complex by $\gsimphom{\bullet}{X,A}$ and call them the \emph{telative homology groups of the pair $(X,A)$.} 
\end{definition}

\begin{definition}
\label{def:relative_cubical_homology}
Let $X$ be a closure space and let $A$ be a subspace. Then note that $\gcubechain{n}{A}\subseteq \gcubechain{n}{X}$. Let $\gcubechain{n}{X,A}:=\gcubechain{n}{X}/\gcubechain{n}{A}$. Note that the boundary map $\partial_n:\gcubechain{n}{X}\to \gcubechain{n-1}{X}$ descends to this quotient. We thus get a new chain complex of abelian groups $(\gcubechain{\bullet}{X,A},\partial_{\bullet})$. Denote the homology groups of this complex by $\gcubehom{\bullet}{X,A}$ and call them the \emph{relative discrete homology groups of the pair $(X,A)$.} 
\end{definition}

\begin{theorem}
\label{theorem:homology_long_exact_sequence}
Let $X$ be a closure space and let $A$ be a subspace. Then there exist  long exact sequences 
\[\cdots\to \gsimphom{n}{A}\to \gsimphom{n}{X}\to \gsimphom{n}{X,A}\to \gsimphom{n-1}{A}\to\cdots\]
\[\cdots\to \gcubehom{n}{A}\to \gcubehom{n}{X}\to \gcubehom{n}{X,A}\to \gcubehom{n-1}{A}\to\cdots\]
\end{theorem}

\begin{proof}
Consider the short exact sequences of chain complexes:
\[\gsimpchain{\bullet}{A}\to \gsimpchain{\bullet}{X}\to\gsimpchain{\bullet}{X,A}\]
\[\gcubechain{\bullet}{A}\to \gcubechain{\bullet}{X}\to\gcubechain{\bullet}{X,A}\]
By standard arguments in homological algebra these short exact sequence  induce the desired long exact sequences in homology.
\end{proof}

\begin{proposition}
\label{prop:relative_homology_of_good_pairs}
Let $(X,A)$ be a $(J,\times)$-good pair of closure spaces. Then the quotient map $q:(X,A)\to (X/A,A/A)$ induces isomorphisms $q_*:\gsimphom{n}{X,A}\to \gsimphom{n}{X/A,A/A}\cong \gsimphomred{n}{X/A}$, $q_*:\gcubehom{n}{X,A}\to \gcubehom{n}{X/A,A/A}\cong \gcubehomred{n}{X/A}$ for all $n$.
\end{proposition}

\begin{proof}
Let $V$ be a neighborhood in $X$ that $(J,\times)$ deformation retracts to $A$. Consider the commutative diagram 
\begin{center}
\begin{tikzcd}
\gsimphom{n}{X,A}\arrow[r]\arrow[d,"q_*"]&\gsimphom{n}{X,V}\arrow[d,"q_*"]&\gsimphom{n}{X-A,V-A}\arrow[l]\arrow[d,"q_*"]\\
\gsimphom{n}{X/A,A/A}\arrow[r]&\gsimphom{n}{X/A,V/A}&\gsimphom{n}{X/A-A/A,V/A-A/A}\arrow[l]
\end{tikzcd}
\end{center}
The upper-left horizontal map is an isomorphism, since in the long exact sequence of the triple $(X,V,A)$, the groups $\gsimphom{n}{V,A}$ are $0$ since $V$ $(J,\times)$ deformation retracts to $A$ and thus induces a $(J,\times)$-homotopy equivalence of pairs $(V,A)\sim_{(J,\times)}(A,A)$ and $\gsimphom{n}{A,A}=0$. This deformation retraction  induces a deformation retraction of $V/A$ to $A/A$ and thus, by the same argument, the bottom-left map is an isomorphism. The other two horizontal maps are isomorphisms by excision (\cref{theorem:excision}). The right vertical map $q_*$ is an isomorphism since $q$ is a homeomorphism on the complement of $A$ in $X$. From the commutativity of the diagram the middle vertical map $q_*$ is then also and isomorphism, and thus the left vertical map $q_*$ is an isomorphism. The cubical case is similar.
\end{proof}

From \cref{theorem:homology_long_exact_sequence,prop:relative_homology_of_good_pairs} we get the following.

\begin{corollary}
\label{corollary:homology_of_quotient_by_closed_subspace}
Let $X$ be a closure space and let $A$ be a non-empty closed subspace of $X$, that is a $(J,\times)$ deformation retract of some neighborhood in $X$. Then, there exist exact sequences
\[\cdots \to \gsimphomred{n}{A}\xrightarrow{i_*}\gsimphomred{n}{X}\xrightarrow{j_*}\gsimphomred{n}{X/A}\xrightarrow{\partial}\gsimphomred{n-1}{A}\xrightarrow{i_*}\gsimphomred{n-1}{X}\to\cdots\]
\[\cdots \to \gcubehomred{n}{A}\xrightarrow{i_*}\gcubehomred{n}{X}\xrightarrow{j_*}\gcubehomred{n}{X/A}\xrightarrow{\partial}\gcubehomred{n-1}{A}\xrightarrow{i_*}\gcubehomred{n-1}{X}\to\cdots\]
where $i:A\to X$ is the inclusion and $j:X\to X/A$ is the quotient map.
\end{corollary}

\subsection{Eilenberg-Steenrod axioms}
Here we define a $(J,\otimes)$ Eileberg-Steenrod homology theory for closure spaces. We show existence of various such theories. 

\begin{definition}
\label{def:indiscrete_eilenberg_steenrod_homology_theory}
A \emph{$(J,\otimes)$ homology theory of closure spaces} consists of the following:
\begin{itemize}
\item A sequence of functors $H_n(-)$ from $\mathbf{Cl}$ to the category of abelian groups.
\item A natural transformation $\partial:H_n((X,A)\to H_{n-1}(A)$.
\end{itemize}
These functors are subject to the following axioms:
\begin{itemize}
\item[1)] (homotopy) $\sim_{(J,\otimes)}$ equivalent maps induce the same map in homology. 
\item[2)] (excision) Let $\{A,B\}$ be an interior cover of $X$. Then the inclusion $(B,A\cap B)\to (X,A)$ induces isomorphisms $H_n(B,A\cap B)\cong (X,A)$, for all $n$.
\item[3)] (dimension) Let $X$ be the one point space. Then $H_n(X,\emptyset)=0$, for all $n\ge 1$.
\item[4)] (long exact sequence) Each pair $((X,A)$ induces a long exact sequence via the inclusion maps $i:A\to X$ and $j:(X,\emptyset)\to (X,A)$:
\[\cdots\to H_n(A)\to H_n(X)\to H_n(X,A)\to H_{n-1}(A)\to \cdots\]
\end{itemize}
\end{definition}

From \cref{theorem:cubical_homotopy_invariance,theorem:simplicial_homotopy_invariance,theorem:homology_long_exact_sequence,theorem:excision,example:homology_of_one_point_space} we have the following result.

\begin{theorem}
\label{theorem:eilenberg_steenrod_existence}
Let $J\in \{I,J_1,J_+\}$. The homology groups $\gsimphom{n}{-}$ and $\gcubesimphom{n}{-}$ are a $(J,\times)$ homology theory of closure spaces.
\end{theorem}

\begin{proposition}
\label{prop:partial_order_between_eilenberg_steenrod_theories}
Let $\otimes_1\le \otimes_2$ be product operations and let $J\le K$ be intervals for $\otimes_2$. If $H_n$ is a $(J,\otimes_1)$ homology theory, then $H_n$ is also a $(K,\otimes_2)$ homology theory. In particular, if $H_n$ is a homology theory for an arbitrary $(J,\otimes)$, then $H_n$ is a homology theory for $(J_1,\times)$.
\end{proposition}

\begin{proof}
Suppose that $H_n$ is a $(J,\otimes_1)$ homology theory. To show that $H_n$ is a $(K,\otimes_2)$ homology theory, note that the only axiom that needs to be checked is the homotopy axiom. To that end, suppose $f\sim_{(K,\otimes_2)}g:X\to Y$. By \cref{prop:order_between_homotopies}, we also have that $f\sim_{(J,\otimes_1)}g$, and since $H_n$ is a $(J,\otimes_1)$ homology theory it follows that $f_*=g_*:H_n(X)\to H_n(Y)$. The statement that every $(J,\otimes)$ homology theory, for an arbitrary product operation $\otimes$ and interval $J$ for $\otimes$, is a $(J_1,\times)$ homology theory then follows from the maximility of the $\sim_{(J_1,\times)}$ homotopy relation from \cref{prop:homotopy_implications}.
\end{proof}

\begin{remark}
\label{remark:non_uniqueness_of_eilenberg_steenrod_theories}
Recall that in the case of topological spaces, on the subcategory of CW complexes, or spaces that have the homotopy type of a CW complex, the Eilenberg-Steenrod axioms imply uniqueness of homology. From \cref{theorem:eilenberg_steenrod_existence,prop:partial_order_between_eilenberg_steenrod_theories}, we have that $\indhom{n}{-}$, $\dihom{n}{-}$ and $\tophom{n}{-}$ are all $(J,\times)$ homology theories for closure spaces. However, there are examples of closure spaces on which these are all distinct homology theories \cite{bubenik2021applied}. Thus, an open question is whether or not there exist full subcategories of closure spaces, on which some of the Eilenberg-Steenrod $(J,\otimes)$ axioms introduced here imply uniqueness. Pursuing  this question is outside the scope of this paper. 
\end{remark}

\subsection{Universal Coefficients theorem}
Let $C_{\bullet}$ be a chain complex of abelian groups. For an abelian group $G$, we can define a new chain complex $C_{\bullet}\otimes_{\mathbb{Z}} G$ by specifying $(C_{\bullet}\otimes_{\mathbb{Z}}G)_n=C_n\otimes_{\mathbb{Z}}G$ and $\partial_n^{C_{\bullet}\otimes_{\mathbb{Z}}G}=\partial_n^C\otimes_{\mathbb{Z}}\cat{1}_G$. We call the homology groups $H_n(C_{\bullet}\otimes_{\mathbb{Z}}G)$ the homology of $C_{\bullet}$ with coefficients in $G$. If $C_{\bullet}$ is one of $\gsimpchain{\bullet}{X}$ or $\gcubechain{\bullet}{X}$, we will denote $H_n(C_{\bullet}\otimes_{\mathbb{Z}}G)$ by $H_n(\gsimpchain{\bullet}{X};G)$ or $H_n(\gcubechain{\bullet}{X};G)$, respectively. From the formalism of homological algebra, whenever we have a chain complex of free abelian groups and we change its coefficients, we can recover the homology groups with coefficients via a universal coefficients theorem \cite{weibel1995introduction}. In particular, for closure spaces we have the following, since $\gsimpchain{\bullet}{X}$ and $\gcubechain{\bullet}{X}$ are complexes of free abelian groups.

\begin{theorem}[UCT for homology]
\label{theorem:uct_theorem_homology}
Let $X$ be a closure space, let $G$ be an abelian group. Then for all $n$ there exist natural short exact sequences
\[H_n(\gsimpchain{\bullet}{X})\otimes_{\mathbb{Z}}G\to H_n(\gsimpchain{\bullet}{X};G)\to\emph{Tor}(H_{n-1}(\gsimpchain{\bullet}{X},G),\]
\[H_n(\gcubechain{\bullet}{X})\otimes_{\mathbb{Z}}G\to H_n(\gcubechain{\bullet}{X};G)\to\emph{Tor}(H_{n-1}(\gcubechain{\bullet}{X},G),\]
which split, though not naturally.
\end{theorem}

\subsection{K\"unneth theorem}
Here we show Eilenberg-Zilber theorems for some singular chain complexes of closure spaces. An immediate consequence are the K\"unneth theorems for some homology groups of closure spaces. 

\begin{definition}
\label{def:tensor_product_of_chain_complexes}
Let $C$ and $D$ be chain complexes of abelian groups. The \emph{tensor product complex $C\otimes D$,} of $C$ and $D$ is the chain complex defined by  $(C\otimes D)_n=\bigoplus_{i+j=n}(C_i\otimes D_j)$ with differential defined on generators by $\partial(a\otimes b) :=\partial_C a\otimes b+ (-1)^ia\otimes \partial_Db$,  $a\in C_i, b\in D_j$ and extending bilinearly.  
\end{definition}

Let $X$ and $Y$ be closure spaces.

\begin{lemma}
\label{lemma:product_of_complexes_of_contractible spaces}
If $X$ and $Y$ are $(J,\times)$-contractible, then $\tilde{H}_n(\gsimpchain{\bullet}{X}\otimes \gsimpchain{\bullet}{Y})=0$, for all $n$.
\end{lemma}

\begin{proof}
Let $m\ge 1$ be such that $H:X\times J^{*m}\to X$ is a homotopy between the identity map $H(-,0)=\cat{1}_X$ on $X$ and a constant map $H(-,m)=a$ with image $a(X)=x_0$ for some $x_0\in X$. Let $\sigma:\gsimp{n}\to X$. Using $H$, we can define a ``prism operator" $P_n:\gsimpchain{n}{X}\to \gsimpchain{n+1}{X}$, in the same way as in the proof of \cref{theorem:simplicial_homotopy_invariance}, which will be a chain homotopy between $\cat{1}_{X*},a_*:\gsimpchain{\bullet}{X}\to \gsimpchain{\bullet}{X}$. Let $\sigma:\gsimp{n}\to X$. Then, from the definition of $P$ we have that $\partial P(\sigma)=\sigma- D(\partial \sigma)$ if $n\ge 1$ and $\partial P(\sigma)=\sigma- x_0$ if $n=0$. Therefore, we have that $\partial P+P\partial=Id-\eps$ where $\eps$ is $0$ nonzero degrees and $\eps(\sum_i n_i\sigma_i)=\sum_i n_ix_0$ can be identified with the augmentation map in degree $0$.

Since by assumption, $X$ and $Y$ are $(J,\times)$ contractible, we have such chain homotopies for both $\gsimpchain{\bullet}{X}$ and $\gsimpchain{\bullet}{Y}$. Define $L_n:(\gsimpchain{\bullet}{X}\otimes \gsimpchain{\bullet}{Y})_n\to (\gsimpchain{\bullet}{X}\otimes \gsimpchain{\bullet}{Y})_{n+1}$ by 
\[L(\sigma\otimes \tau)=P(\sigma)\otimes \eps(\tau)+(-1)^{|\sigma|}\sigma\otimes P(\tau)\]
and extending linearly. A direct computation shows that $L$ is a chain homotopy between $Id\otimes Id$ and $\eps\otimes \eps$. Thus, $\tilde{H}_n(\gsimpchain{\bullet}{X}\otimes \gsimpchain{\bullet}{Y})$ is equal to the reduced homology of the image under $\eps \otimes \eps$, which is a trivial chain complex with just $\mathbb{Z}$ in degree $0$. Hence the result follows.
\end{proof}

\begin{theorem}[\textbf{Theorem 1.3} (Eilenberg-Zilber)]
\label{theorem:eilenberg_zilber}
Let $X$ and $Y$ be closure spaces. Then there are natural chain homotopy equivalences $\gsimpchain{\bullet}{X\times Y}\simeq \gsimpchain{\bullet}{X}\otimes \gsimpchain{\bullet}{Y}$ and $\gcubesimpchain{\bullet}{X\times Y}\simeq \gcubesimpchain{\bullet}{X}\otimes \gcubesimpchain{\bullet}{Y}$.
\end{theorem}

\begin{proof}
We present the arguments for the simplicial case. The cubical case is similar. Let $\cat{C}$ be the category of pairs of closure spaces $(X,Y)$ and morphisms $(f,g):(X_1,X_2)\to (Y_1,Y_2)$ where $f:X_1\to Y_1$ and $g:X_2\to Y_2$ are continuous maps of closure spaces. Let $F,G:\cat{C}\to \cat{Ch_{\ge 0}}(\cat{Ab})$ be the functors $F(X,Y):=\gsimpchain{\bullet}{X\times Y}$ and $G(X,Y):=\gsimpchain{\bullet}{X}\otimes \gsimpchain{\bullet}{Y}$.  Let $\mathscr{M}$ be the models of $\cat{C}$ (\cref{section:acyclic_models}) consisting of pairs $(\gsimp{n},\gsimp{m})_{n,m\in \mathbb{N}}$. Note that $\gsimp{n}$ is $(J,\times)$-contractible by \cref{example:simplices_are_contractible} for each $n\ge 0$. Therefore so is $\gsimp{n}\times \gsimp{m}$ for $n,m\in \mathbb{N}$ \cite[Corollary 5.34]{bubenik2021applied}. In particular, $F$ is acyclic on $\mathscr{M}$ (\cref{def:free_functor}). By \cref{example:simplices_are_contractible,lemma:product_of_complexes_of_contractible spaces} it follows that $G$ is also acyclic on $\mathscr{M}$. Let $D_n\in F_n((\gsimp{n},\gsimp{n}))$ be the $J$ singular $n$-chain on $\gsimp{n}\times \gsimp{n}$ given by the diagonal map $D_n:\gsimp{n}\to \gsimp{n}\times \gsimp{n}$. Note that $F_n((X,Y))$ is by definition free on $\cat{Cl}(\gsimp{n},X\times Y)\cong \cat{Cl}(\gsimp{n},X)\times \cat{Cl}(\gsimp{n},Y)=\cat{C}((\gsimp{n},\gsimp{n}),(X,Y))$. The map $\cat{C}((\gsimp{n},\gsimp{m}),(X,Y))\to F_n(X,Y)$ sends $\sigma:\gsimp{n}\to X$, $\tau:\gsimp{n}\to Y$ to the pullback of $\sigma\times \tau:\gsimp{n}\times \gsimp{n}\to X\times Y$ via $D_n$. Thus $F$ is free on on $\mathscr{M}$ (\cref{def:free_functor}). Finally observe that $G_n$ is by definition free on $\coprod_{i+j=n}\cat{Cl}(\gsimp{i},X)\times \cat{Cl}(\gsimp{j},Y)\cong  \coprod_{i+j=n}\cat{C}((\gsimp{i}\times \gsimp{j}),(X\times Y))$. It follows that $G$ is also free on $\mathscr{M}$. Therefore, by \cref{theorem:acyclic_model_theorem} the statement of the theorem follows.
\end{proof} 
 
\begin{corollary}
\label{corollary:homology_of_product}
For all $n\ge 0$, $\gsimphom{n}{X\times Y}\cong H_n(\gsimpchain{\bullet}{X}\otimes \gsimpchain{\bullet}Y)$ and  $\gcubesimphom{n}{X\times Y}\cong H_n(\gcubesimpchain{\bullet}{X}\otimes \gcubesimpchain{\bullet}Y)$.
\end{corollary}

\begin{corollary}[K\"unneth Theorem]
\label{corollary:kunneth}
Let $X$ and $Y$ be closure spaces. Let $R$ be a principal ideal domain. Then for all $n$ there are natural short exact sequences 
\begin{gather*}
\bigoplus\limits_i \gsimphom{i}{X;R}\otimes_R \gsimphom{n-i}{Y;R}\to \gsimphom{n}{X\times Y;R}
\to\bigoplus\limits_i\text{Tor}_i(\gsimphom{i}{X;R},\gsimphom{n-i-1}{Y;R})
\\
\bigoplus\limits_i \gcubesimphom{i}{X;R}\otimes_R \gcubesimphom{n-i}{Y;R}\to \gcubesimphom{n}{X\times Y;R}
\to\bigoplus\limits_i\text{Tor}_i(\gcubesimphom{i}{X;R},\gcubesimphom{n-i-1}{Y;R})
\end{gather*}
which split, though not naturally.
\end{corollary}

\begin{remark}
\label{remark:kunneth_theorem_homology}
The question whether the Eilenberg-Zilber theorem or the K\"unneth theorem are valid for the functors $\gcubechain{\bullet}{-}$ and $\gcubehom{\bullet}{-}$, when $J\in \{I,J_+,J_1\}$ and $\otimes=\boxdot$ remains open. The reason the arguments in the proof of \cref{theorem:eilenberg_zilber} do not apply is the failure of diagonal maps $I\to I\boxdot I$, $J_+\to J_+\boxdot J_+$ and $J_1\to J_1\boxdot J_1$ to be continuous, hence showing one of the functors in the proof of \cref{theorem:eilenberg_zilber} is free on a collection of models is a harder problem.
\end{remark}

\section{Cohomology of Closure Spaces}
\label{section:cohomology}
Let $(C_{\bullet},\partial_{\bullet})$ be a chain complex of abelian groups and let $G$ be an abelian group. By applying the contravariant functor $\Hom_{\mathbb{Z}}(-,G)$ to each $C_n$ and each $\partial_n:C_n\to C_{n-1}$ we get a cochain complex $(C^{\bullet},\delta^{\bullet})$, where $C^n=\Hom_{\mathbb{Z}}(C_n,G)$ and $\delta^n:C^{n-1}\to C^n$ is the induced map from $\partial_n$ by $\Hom_{\mathbb{Z}}(-,G)$. 

Let $X$ be a closure space, $\otimes=\{\times,\boxdot\}$ and $J\in\{J_1,J_+,I\}$.  If $C_{\bullet}$ is one of the chain complexes $\gsimpchain{\bullet}{X}$ or $\gcubechain{\bullet}{X}$, we denote $C^{\bullet}$ by $\gsimpcochain{\bullet}{X;G}$ or $\gcubecochain{\bullet}{X;G}$, respectively. We denote the cohomology groups of these cochain complexes by $\gsimpcohom{\bullet}{X;G}$ and $\gcubecohom{\bullet}{X;G}$ respectively.

In this section we investigate some properties of these cohomology theories of closure spaces. The hard work was done in \cref{section:homology} and almost everything in this section follows by applying classical arguments in homological algebra to results in \cref{section:homology}. Throughout this section, $J\in \{J_1,J_+,I\}$ and $\otimes\in \{\times, \boxdot\}$.

\begin{remark}
\label{remark:chain_homotopy_induce_cochain_homotopies}
If $f:C_{\bullet}\to D_{\bullet}$ is a map of chain complexes, then it is well known that applying the contraviant functor $\Hom_{\mathbb{Z}}(-,G)$ induces a cochain map $f^*:D^{\bullet}\to C^{\bullet}$. Furthermore, suppose that $f,g:C_{\bullet}\to D_{\bullet}$ are chain homotopic. That is, there exist a chain homotopy $P$ such that $f-g=\partial^D_{\bullet} P+P\partial^C_{\bullet}$. Then, applying $\Hom_{\mathbb{Z}}(-,G)$ to both sides of the last equation yields $f^*-g^*=P^*\delta_D^{\bullet}+\delta_C^{\bullet}P^*$. That is, $P^*$ is a cochain homotopy between $f^*-g^*$. We thus immediately have the following result that follows from \cref{theorem:simplicial_and_cubical_isomorphism}
\end{remark}

\begin{theorem}
\label{theorem:simplicial_and_cubical_cohom_isomorphism}
There exist chain homotopies $\gsimpcochain{\bullet}{X;G}\simeq \gcubesimpcochain{\bullet}{X;G}$ for any closure space $X$.
\end{theorem}

\begin{proposition}
\label{prop:cohomology_decomposition_into_path_components}
Let $X$ be a closure space. Then $\gcubecohom{n}{X}\cong \prod_{\alpha}\gcubecohom{n}{X_{\alpha}}$ and $\gsimpcohom{n}{X}\cong \prod_{\alpha}\gsimpcohom{n}{X_{\alpha}}$, where $\{X_{\alpha}\}_{\alpha\in A}$ are the $J$-path components of $X$.
\end{proposition}

\begin{proof}
Recall that in the proof of \cref{prop:homology_decomposition_into_path_components} we argued that each $\gcubechain{n}{X}$ and $\gsimpchain{n}{X}$ splits into direct sum of their subgroups $\gcubechain{n}{X_{\alpha}}$ and $\gsimpchain{n}{X_{\alpha}}$, respectively. The result then follows from the fact that $\text{Hom}_{\mathbb{Z}}(\oplus_{\alpha} H_{\alpha},G)=\prod_i\text{Hom}_{\mathbb{Z}}(H_{\alpha},G)$ for any family of abelian groups $\{H_{\alpha}\}_{\alpha\in A}$. The coboundary maps of the two cochain complexes will respect this decomposition, hence the corresponding cohomology groups also split into products on the $J$-path components accordingly. 
\end{proof}

\subsection{Universal Coefficients theorems}
Suppose we have a chain complex $C_{\bullet}$ of free abelian groups and we consider its cochain complex $C^{\bullet}$ with coefficients in an abelian group $G$. By formalism of homological algebra, it is possible to calculate cohomology of $C^{\bullet}$ in terms of the homology of $C_{\bullet}$ using universal coefficients theorems \cite{weibel1995introduction}. In particular, for closure spaces we have the following results, since $\gsimpchain{\bullet}{X}$ and $\gcubechain{\bullet}{X}$ are complexes of free abelian groups.

\begin{theorem}[UCT for cohomology]
\label{theorem:uct_theorem_cohomology}
Let $X$ be a closure space, let $G$ be an abelian group. Then for all $n$ there exist natural short exact sequences
\[\emph{Ext}(H_{n-1}(\gsimpchain{\bullet}{X};G)\to \gsimpcohom{\bullet}{X;G}\to \emph{Hom}_{\mathbb{Z}}(\gsimphom{n}{X},G),\]
\[\emph{Ext}(H_{n-1}(\gcubechain{\bullet}{X};G)\to \gcubecohom{\bullet}{X;G}\to \emph{Hom}_{\mathbb{Z}}(\gcubehom{n}{X},G),\]
which split, though not naturally.
\end{theorem}

From this, and the five lemma one can deduce the following.

\begin{corollary}
\label{corollary:chain_map_inducing_isomorphism}
If a chain map $f:\gsimpchain{\bullet}{X}\to \gsimpchain{\bullet}{Y}$ ($f:\gcubechain{\bullet}{X}\to \gcubechain{\bullet}{Y}$) induces an isomorphism on homology, then the cochain map $f^*:\gsimpcochain{\bullet}{Y;G}\to \gsimpcochain{\bullet}{X;G}$ ($f^*:\gcubecochain{\bullet}{Y;G}\to \gcubecochain{\bullet}{X;G}$) induces an isomorphism on cohomology, for any abelian group $G$.
\end{corollary}

From \cref{prop:cohomology_decomposition_into_path_components,theorem:uct_theorem_cohomology,prop:0_homology_counts_path_components} we also get the following result.

\begin{proposition}
\label{prop:0_cohomology_counts_path_components}
Let $X$ be a closure space. Then $\gcubecohom{0}{X}\cong \gsimpcohom{0}{X}\cong \prod_{a\in \pi_0^J(X)}\mathbb{Z}$.
\end{proposition}

\subsection{Homotopy Invariance}
Let $f:X\to Y$ be a map of closure spaces. Recall that $f$ induces chain maps $f_{\#}:\gsimpchain{\bullet}{X}\to \gsimpchain{\bullet}{Y}$ and $f_{\#}:\gcubechain{\bullet}{X}\to \gcubechain{\bullet}{Y}$. Applying the contravariant functor $\Hom_{\mathbb{Z}}(-,G)$ to these complexes then induces cochain maps $f^{\#}:\gsimpcochain{\bullet}{Y;G}\to\gsimpcochain{\bullet}{X;G}$ and $f^{\#}:\gcubecochain{\bullet}{Y;G}\to \gcubecochain{\bullet}{X;G}$, respectively. Furthermore, this map induces maps  on cohomology $f^n:\gsimpcohom{n}{Y;G}\to \gsimpcohom{n}{X;G}$ and $f^n:\gcubecohom{n}{Y;G}\to \gcubecohom{n}{X;G}$.
If $f,g:X\to Y$ are continuous and $f\sim_{(J,\otimes)}g$, one can construct a chain homotopy $P$ between $f_{\#},g_{\#}:\gcubechain{\bullet}{X}\to \gcubechain{\bullet}{Y}$ \cite[Theorem 6.6]{bubenik2021applied}. That is $\partial P+P\partial=f_{\#}-g_{\#}$. Similarly, we have a chain homotopy in the simplicial case that was constructed in the proof of \cref{theorem:simplicial_homotopy_invariance}. Applying the contravariant functor, in both cases, $\Hom_{\mathbb{Z}}(-,G)$ to the chain complexes in question, produces cochain homotopies as observed in \cref{remark:chain_homotopy_induce_cochain_homotopies}. We thus have the following result about homotopy invariance of cohomology.

\begin{theorem}
\label{theorem:homotopy_invariance_of_cohomology}
Suppose $f\sim_{(J,\otimes)}g:X\to Y$. Then $f^*=g^*:\gcubecohom{\bullet}{Y;G}\to \gcubecohom{\bullet}{X;G}$. If $f\sim_{(J,\times)}g$ then $f^*=g^*:\gsimpcohom{\bullet}{Y;G}\to \gsimpcohom{\bullet}{X;G}$.
\end{theorem}

\subsection{Relative Cohomology}
Recall that for a pair of closure spaces $(X,A)$ we had the following short exact sequences of chain complexes
\begin{gather*}
0\to \gsimpchain{\bullet}{A}\xrightarrow{i}\gsimpchain{\bullet}{X}\xrightarrow{j}\gsimpchain{\bullet}{X,A}\to 0,\\
0\to \gcubechain{\bullet}{A}\xrightarrow{i}\gcubechain{\bullet}{X}\xrightarrow{j}\gcubechain{\bullet}{X,A}\to 0.
\end{gather*}
The contravariant functor $\Hom_{\mathbb{Z}}(-,G)$ is known to be left exact from classical module theory. Thus, applying it to the above sequences yields the following left exact sequences
\begin{gather*}
0\to \gsimpcochain{\bullet}{X,A;G}\xrightarrow{j^*}\gsimpcochain{\bullet}{X;G}\xrightarrow{i^*}\gsimpcochain{\bullet}{A;G},\\    
0\to \gcubecochain{\bullet}{X,A;G}\xrightarrow{j^*}\gcubecochain{\bullet}{X;G}\xrightarrow{i^*}\gcubecochain{\bullet}{A;G}.
\end{gather*}
where we define $\gsimpcochain{n}{X,A;G}=\Hom_{\mathbb{Z}}(\gsimpchain{n}{X},G)$ and $\gcubecochain{n}{X,A;G}=\Hom_{\mathbb{Z}}(\gcubechain{n}{X},G)$ for all $n$. It turns out that these sequences are not only right exact, but are in fact exact. To show this, it suffices to argue that the chain maps $i^*$ are epimorphisms, that is for all $n$, $i^{*n}$ is a surjective group homomorphisms. Note that $i^{*n}$ restricts cochains on $X$ to cochains on $A$. Thus, given a group homomorphisms $\varphi:\gsimpchain{n}{X}\to G$, $i^{*n}\varphi$ is the restriction of $\varphi$ to the subgroup $\gsimpchain{n}{A}\subseteq \gsimpchain{n}{X}$. If $\psi:\gsimpchain{n}{A}\to G$ is a cochain on $A$, we can extend it to a cochain on $X$ by extending by $0$ on those chains that are in $X$ but not in $A$. Thus $i^{*n}$ is surjective for all $n$. Same arguments are applicable in the cubical case. Thus, since we have short exact sequences of cochains, standard arguments using the snake lemma guarantee the existence of a long exact sequence on cohomology. The formal statement is as follows.

\begin{theorem}
\label{theorem:relative_cohomology_long_exact_sequence}
For a closure space pair $(X,A)$ we have the following long exact sequences 
\begin{gather*}
\cdots\to \gsimpcohom{n}{X,A;G}\xrightarrow{j^*}\gsimpcohom{n}{X;G}\xrightarrow{i^*}\gsimpcohom{n}{A;G}\xrightarrow{\delta}\gsimpcohom{n+1}{X,A;G}\to\cdots,\\
\cdots\to \gcubecohom{n}{X,A;G}\xrightarrow{j^*}\gcubecohom{n}{X;G}\xrightarrow{i^*}\gcubecohom{n}{A;G}\xrightarrow{\delta}\gcubecohom{n+1}{X,A;G}\to\cdots.
\end{gather*}
\end{theorem}
 
\subsection{Excision and Mayer-Vietoris}
Let $\mathcal{U}$ be an interior cover for a closure space $X$. Recall the chain complexes $\gsimpchain{\bullet}{X}^{\mathcal{U}}$ and $\gcubechain{\bullet}{X}^{\mathcal{U}}$. Applying the functor $\Hom_{\mathbb{Z}}(-,G)$ to these we get cochain complexes $\gsimpcochain{\bullet}{X;G}^{\mathcal{U}}$ and $\gcubecochain{\bullet}{X;G}^{\mathcal{U}}$, respectively. In \cref{prop:chain_of_interior_cover} where we showed that $\topchain{\bullet}{X}^{\mathcal{U}}$ and $\topchain{\bullet}{X}$ are chain homotopic and that $\topcubechain{\bullet}{X}^{\mathcal{U}}$ and $\topcubechain{\bullet}{X}$. We also showed the stronger statements that $\gsimpchain{\bullet}{X}=\gsimpchain{\bullet}{X}^{\mathcal{U}}$ and $\gcubesimpchain{\bullet}{X}=\gcubesimpchain{\bullet}{X}^{\mathcal{U}}$ when $J\in \{J_1,J_+\}$. From \cref{remark:chain_homotopy_induce_cochain_homotopies} we thus get the following.

\begin{proposition}
\label{prop:cochain_of_interior_cover}
$\gsimpcochain{\bullet}{X;G}^{\mathcal{U}}$ is chain homotopy equivalent to $\gsimpcochain{\bullet}{X;G}$ and $\gcubesimpcochain{\bullet}{X;G}^{\mathcal{U}}$ is chain homotopy equivalent to $\gcubesimpcochain{\bullet}{X;G}$. Moreover, $\gsimpcochain{\bullet}{X;G}^{\mathcal{U}}=\gsimpcochain{\bullet}{X;G}$ and\\ $\gcubesimpcochain{\bullet}{X;G}^{\mathcal{U}}=\gcubesimpcochain{\bullet}{X;G}$ when $J\in \{J_+,J_1\}$.
\end{proposition}

\begin{theorem}[Excision]
\label{theorem:excision_cohomology}
Let $(X,c)$ be a closure space. Let $Z\subseteq A\subseteq X$ be such that $c(Z)\subseteq i(A)$. Then, the inclusion $(X-Z,A-Z)\to (X,A)$ induces isomorphisms $\gsimpcohom{n}{X-Z,X-A;G}\cong \gsimpcohom{n}{X,A;G}$ and $\gcubesimpcohom{n}{X-Z,X-A;G}\cong \gcubesimpcohom{n}{X,A;G}$ for all $n$. Equivalently, for an interior cover $\{A,B\}$ of $(X,c)$, the inclusion $(B,A\cap B)\to (X,A)$ induces isomorphisms $\gsimpcohom{n}{B,A\cap B;G}\cong \gsimpcohom{n}{X,A;G}$ and $\gcubesimpcohom{n}{B,A\cap B;G}\cong \gcubesimpcohom{n}{X,A;G}$ for all $n$.
\end{theorem}

\begin{proof}
This follows from \cref{theorem:excision,theorem:uct_theorem_cohomology} and the five lemma. 
\end{proof}

\subsection{Eilenberg-Steenrod axioms for cohomology}
Here we define a $(J,\otimes)$ Eilenberg-Steenrod cohomology theory for closure spaces. We show existence of various such theories. 

\begin{definition}
\label{def:indiscrete_eilenberg_steenrod_homology_theory}
A \emph{$(J,\otimes)$ cohomology theory of closure spaces} consists of the following:
\begin{itemize}
\item A sequence of contravariant functors $H_n(-)$ from $\mathbf{Cl}$ to the category of abelian groups.
\item A natural transformation $\partial:H^n((X,A)\to H^{n-1}(A)$.
\end{itemize}
These functors are subject to the following axioms:
\begin{itemize}
\item[1)] (homotopy) $\sim_{(J,\otimes)}$ equivalent maps induce the same map in cohomology. 
\item[2)] (excision) Let $\{A,B\}$ be an interior cover of $X$. Then the inclusion $(B,A\cap B)\to (X,A)$ induces isomorphisms $H^n(B,A\cap B)\cong H^n(X,A)$, for all $n$.
\item[3)] (dimension) Let $X$ be the one point space. Then $H^n(X,\emptyset)=0$, for all $n\ge 1$.
\item[4)] (long exact sequence) Each pair $((X,A)$ induces a long exact sequence via the inclusion maps $i:A\to X$ and $j:(X,\emptyset)\to (X,A)$:
\[\cdots\to H^n(A)\to H_n(X)\to H^n(X,A)\to H^{n+1}(A)\to \cdots\]
\end{itemize}
\end{definition}

From \cref{theorem:cubical_homotopy_invariance,theorem:simplicial_homotopy_invariance,theorem:homology_long_exact_sequence,theorem:excision,example:homology_of_one_point_space} we have the following result.

\begin{theorem}
\label{theorem:eilenberg_steenrod_existence_cohomology}
Let $J\in \{I,J_1,J_+\}$. 
The cohomology groups $\gsimpcohom{n}{-;G}$ and $\gcubesimpcohom{n}{-;G}$ are a $(J,\times)$ homology theory of closure spaces.
\end{theorem}

\begin{proposition}
\label{prop:partial_order_between_eilenberg_steenrod_theories}
Let $\otimes_1\le \otimes_2$ be product operations and let $J\le K$ be intervals for $\otimes_2$. If $H^n$ is a $(J,\otimes_1)$ cohomology theory, then $H^n$ is also a $(K,\otimes_2)$ cohomology theory. In particular, if $H^n$ is a cohomology theory for an arbitrary $(J,\otimes)$, then $H^n$ is a cohomology theory for $(J_1,\times)$.
\end{proposition}

\begin{proof}
Suppose that $H^n$ is a $(J,\otimes_1)$ cohomology theory. To show that $H^n$ is a $(K,\otimes_2)$ cohomology theory, note that the only axiom that needs to be checked is the homotopy axiom. To that end, suppose $f\sim_{(K,\otimes_2)}g:X\to Y$. By \cref{prop:order_between_homotopies}, we also have that $f\sim_{(J,\otimes_1)}g$, and since $H^n$ is a $(J,\otimes_1)$ cohomology theory it follows that $f_*=g_*:H^n(Y)\to H^n(X)$. The statement that every $(J,\otimes)$ cohomology theory, for an arbitrary product operation $\otimes$ and interval $J$ for $\otimes$, is a $(J_1,\times)$ cohomology theory then follows from the maximility of the $\sim_{(J_1,\times)}$ homotopy relation from \cref{prop:homotopy_implications}.
\end{proof}

\begin{remark}
\label{remark:non_uniqueness_of_eilenberg_steenrod_theories_cohomology}
Recall that in the case of topological spaces, on the subcategory of CW complexes, or spaces that have the homotopy type of a CW complex, the Eilenberg-Steenrod axioms imply uniqueness of cohomology. From \cref{theorem:eilenberg_steenrod_existence_cohomology,prop:partial_order_between_eilenberg_steenrod_theories}, we have that for $J\in \{I,J_1,J_+\}$, $\gsimpcohom{n}{-;G}$ and $\gcubesimpcohom{n}{-;G}$ are $(J,\times)$ cohomology theories for closure spaces.
In \cref{remark:non_uniqueness_of_eilenberg_steenrod_theories} we argued for the existence of closure spaces on which we have non-isomorphic homology theories. Taking cohomology of those spaces with coefficients in a field, by \cref{theorem:uct_theorem_cohomology} we will get examples of closure spaces on which these are all distinct cohomology theories. 
Thus, an open question is whether or not there exist full subcategories of closure spaces, on which some of the Eilenberg-Steenrod $(J,\otimes)$ cohomology axioms introduced here imply uniqueness. Pursuing  this question is outside the scope of this paper. 
\end{remark}

\subsection{K\"unneth theorems}
Let $R$ be a commutative ring with unity.
Here we state the Eilenberg-Zilber theorem for some singular simplicial and cubical cochain complexes of closure spaces. An immediate consequence are the K\"unneth theorems for the corresponding cohomology groups. In fact, by \cref{theorem:eilenberg_zilber,remark:chain_homotopy_induce_cochain_homotopies} we immediately have the following theorem.

\begin{theorem}[Eilenberg-Zilber]
\label{theorem:eilenberg_zilber_cohomology}
Let $X$ and $Y$ be closure spaces. Then there are natural chain homotopy equivalences $\gsimpcochain{\bullet}{X\times Y;R}\simeq (\gsimpchain{\bullet}{X}\otimes \gsimpchain{\bullet}{Y};R)^{\bullet}$ and $\gcubesimpcochain{\bullet}{X\times Y;R}\simeq (\gcubesimpchain{\bullet}{X}\otimes \gcubesimpchain{\bullet}{Y};R)^{\bullet}$.
\end{theorem}

\begin{theorem}
\label{theorem:kunneth_cohomology}
There exist natural isomorphisms $\gsimpcohom{n}{X\times Y;R}\cong H^n(\gsimpcochain{\bullet}{X;R}\otimes \gsimpcochain{\bullet}{Y;R})$ and $\gcubesimpcohom{n}{X\times Y;R}\cong H^n(\gcubesimpcochain{\bullet}{X;R}\otimes \gcubesimpcochain{\bullet}{Y;R})$ .
\end{theorem}

\begin{proof}
Let $A:\gsimpcochain{\bullet}{X\times Y;R}\to (\gsimpchain{\bullet}{X}\otimes \gsimpchain{\bullet}{Y};R)^{\bullet}$ and $Z:(\gsimpchain{\bullet}{X}\otimes \gsimpchain{\bullet}{Y};R)^{\bullet}\to \gsimpcochain{\bullet}{X\times Y;R}$ be the maps that realize the chain homotopy from \cref{theorem:eilenberg_zilber_cohomology}. In general it is not true that $(\gsimpchain{\bullet}{X}\otimes \gsimpchain{\bullet}{Y};R)^{\bullet}\cong \gsimpcochain{\bullet}{X;R}\otimes \gsimpcochain{\bullet}{Y;R}$. However, there is a natural monomorphism $i:\gsimpcochain{\bullet}{X;R}\otimes \gsimpcochain{\bullet}{Y;R}\to  (\gsimpchain{\bullet}{X}\otimes \gsimpchain{\bullet}{Y};R)^{\bullet}$ defined by $i(a\otimes b):=(\sigma\otimes \tau\mapsto a(\sigma)b(\tau))$. The chain map $i$ can be checked to be a quasi isomorphism and thus we have an isomorphism for all $n\ge 0$
\[H^{n}(Z)\circ H^{n}(i):H^n(\gsimpcochain{\bullet}{X;R}\otimes \gsimpcochain{\bullet}{Y;R})\cong \gsimpcohom{\bullet}{X\times Y;R}.\]
Similarly, for all $n\ge 0$ we have
\[H^n(\gcubesimpcochain{\bullet}{X;R}\otimes \gcubesimpcochain{\bullet}{Y;R})\cong \gcubesimpcohom{\bullet}{X\times Y;R}.\qedhere\]
\end{proof}

\begin{theorem}{\cite[Theorem 5.5.11]{spanier1989algebraic}}
\label{theorem:finite_hypothesis_for_kunneth}
Let $C$ and $C'$ be nonnegative free chain complexes of abelian groups and $G$ and $G'$ be modules over a principal ideal domain $R$ such that $\text{Tor}_1^R(G,G')=0$ and either $H_k(C)$ or $H_k(C')$ are of finite type or $H_k(C')$ is of finite type and $G'$ is finitely generated. Then there is a natural short exact sequence
\[\bigoplus_{j+k=n}H^j(C;G)\otimes_R H^k(C';G')\to H^n(C\otimes C'; G\otimes_R G')\to\bigoplus_{j+k=n+1} \text{Tor}_1^R(H^j(C;G),H^k(C';G'))\]
and this sequence splits. 
\end{theorem}

Thus, by \cref{theorem:finite_hypothesis_for_kunneth,theorem:eilenberg_zilber_cohomology,theorem:kunneth_cohomology} we have the following theorem.

\begin{theorem}[K\"unneth theorem for singular cohomologies]
\label{theorem:Kunneth_cohomology}
Let $X$ and $Y$ be closure spaces. Suppose that all $\gsimpcohom{\bullet}{Y;R}\cong \gcubesimpcohom{\bullet}{Y;R}$ are finitely generated. Then there are natural short exact sequence
\[\bigoplus_{j+k=n}\gsimpcohom{j}{X;R}\otimes_R \gsimpcohom{j}{Y;R}\xrightarrow{\times^J} \gsimpcohom{n}{X\times Y;R}\to\bigoplus_{j+k=n+1} \text{Tor}_1^R(\gsimpcohom{j}{X;R},\gsimpcohom{k}{Y;R})\]
\[\bigoplus_{j+k=n}\gcubesimpcohom{j}{X;R}\otimes_R \gcubesimpcohom{j}{Y;R}\xrightarrow{\times^{(J,\times)}} \gcubesimpcohom{n}{X\times Y;R}\to\bigoplus_{j+k=n+1} \text{Tor}_1^R(\gcubesimpcohom{j}{X;R},\gcubesimpcohom{k}{Y;R})\]
and these sequences splits.

\end{theorem}

\begin{definition}
\label{def:cup_product}
Let $\smile^{J}:\gsimpcohom{p}{X;R}\otimes_R \gsimpcohom{q}{X;R}\to \gsimpcohom{p+q}{X;R}$ and $\smile^{(J,\times)}:\gcubesimpcohom{p}{X;R}\otimes_R \gcubesimpcohom{q}{X;R}\to \gcubesimpcohom{p+q}{X;R}$ be defined by the compositions
\[\gsimpcohom{p}{X}\times\gsimpcohom{q}{X}\xrightarrow{\times^J} \gsimpcohom{p+q}{X\times X}\xrightarrow{\Delta^*}\gsimpcohom{p+q}{X} \]
\[\gcubesimpcohom{p}{X}\times\gcubesimpcohom{q}{X}\xrightarrow{\times^{(J,\times)}} \gcubesimpcohom{p+q}{X\times X}\xrightarrow{\Delta^*}\gcubesimpcohom{p+q}{X} \]
where the maps $\times^J$ and $\times^{(J,\times)}$ are from the K\"unneth theorem. We call $\smile^J$ and $\smile^{(J,\times)}$ the \emph{cup products on $\gsimpcohom{\bullet}{X;R}$} and $\gcubesimpcohom{\bullet}{X;R}$ respectively. This gives $\gsimpcohom{\bullet}{X;R}$ and $\gcubesimpcohom{\bullet}{X;R}$ graded ring structures.
\end{definition}

\begin{remark}
\label{remark:cup_product_on_inductive_product}
One might wonder why we don't define a cup product on $\gicubecohom{\bullet}{X}$ in a similar fashion. The problem is that the diagonal map $\Delta:X\to X\boxdot X$ is usually not continuous. For example the diagonal maps $\Delta:I\to I\boxdot I,\Delta:J_+\to J_+\boxdot J_+$ and $\Delta:J_1\to J_1\boxdot J_1$ are certainly not. Thus putting a ring structure analogous to the cup product from algebraic topology on $\gicubecohom{\bullet}{X}$ is an open problem for now.
\end{remark}

\subsection*{Acknowledgments}
This material is based upon work supported by, or in part by, the Army Research Laboratory and the Army Research Office under contract/grant number W911NF-18-1-0307.

\appendix

\section{Filters}

Here we recall the definition of a filter. It turns out that neighborhoods of a point in a closure space form a filter. Furthermore, a basis of this filter at each point is sufficient information to reconstruct the closure operation. Thus, as for topological spaces, a point $x$ is in a closure of a set $A$ if and only if all basic neighborhoods of $x$ have a non-empty intersection with $A$ (see \cref{theorem:filter_determines_a_local_base_for_a_closure}).

\begin{definition}
\label{def:filter}
Given a set $S$, a \emph{filter} on $S$ is a nonempty collection $F\subseteq \mathcal{P}(S)$ of subsets of $S$ such that: 
\begin{itemize}
\item[1)] (closed under binary intersection) $A,B\in F\Longrightarrow A\cap B\in F$.
\item[2)] (upward closed) $A\in F, A\subseteq B\Longrightarrow B\in F$.
\end{itemize}
\end{definition}

\begin{definition}
\label{def:base_of_a_filter}
Given a set $S$, a \emph{filter base} is a nonempty collection $B\subseteq \mathcal{P}(S)$ of subsets of $S$ such that:
\begin{itemize}
\item[1)] (downward directed)
  for any $A_1,A_2\in B$, there exists an $A_3\in B$ such that $A_3\subseteq A_1\cap A_2$.
\end{itemize}
Given a filter base $B$, the filter \emph{generated} by $B$ is
obtained by taking the upward closure of the base.
We also say $B$ is a \emph{base} for $F$.
\end{definition}

\begin{theorem}{\cite[Theorem 14.B.3]{vcech1966topological}}
\label{theorem:neighborhood_system_is_a_filter}
Let $\mathcal{U}$ be the neighborhood system of a subset $A$ of a closure space $(X,c)$. Then $\mathcal{U}$ is a filter on $X$ such that $A\subseteq \bigcap \mathcal{U}$.
\end{theorem}

\begin{definition}{\cite[Definition 14.B.4]{vcech1966topological}}
\label{def:base_of_a_neighborhood}
Let $(X,c)$ be a closure space. By \cref{theorem:neighborhood_system_is_a_filter} the neighborhood system of a set $A\subseteq X$ is a filter. A base of this filter is called a \emph{base of the neighborhood system of} $A$ \emph{in} $(X,c)$. We will call a \emph{local base at} $x$ a base of the neighborhood system of $\{x\}$.
\end{definition}

\begin{proposition}{\cite[14.B.5]{vcech1966topological}}
\label{prop:filter_base_properties}
If  $\mathcal{U}_x$ is a local base at $x$ in a closure space $(X,c)$ then the following are true:
\begin{itemize}
\item[(nbd 1)] $\mathcal{U}_x\neq \emptyset$.
\item[(nbd 2)] For all $U\in \mathcal{U}_x$, $x\in U$.
\item[(nbd 3)] For each $U_1,U_2\in \mathcal{U}_x$, there exists a $U\in \mathcal{U}_x$ such that $U\subseteq U_1\cap U_2$.
\end{itemize}
\end{proposition}

\begin{theorem}{\cite[Theorem 14.B.10]{vcech1966topological}}
\label{theorem:filter_determines_a_local_base_for_a_closure}
For each element $x$ of a set $X$, let $\mathcal{U}_x$ be a collection of subsets satisfying the conditions in \cref{prop:filter_base_properties}. Then there exists a unique closure operation $c$ for $X$ such that for all $x\in X$, $\mathcal{U}_x$ is a local base at $x$ in $(X,c)$. More specifically, $c$ can be defined as 
\[c(A):=\{x\in X\,|\,U\in \mathcal{U}_x\Longrightarrow U\cap A\neq \emptyset\}\]
for all $A\subseteq X$.
\end{theorem}

\section{Acyclic models}
\label{section:acyclic_models}
Let $\cat{C}$ be a category and let $\mathscr{M}$ be a class of objects of $\cat{C}$ which we will call \emph{model objects}. Let $\cat{Ab}$ be the category of abelian groups and group homomorphisms and let $\cat{Ch}_{\ge 0}(\cat{Ab})$ be the category of chain complexes of abelian groups that are trivial in negative dimension with morphism the chain maps.

Let $T:\cat{C}\to \cat{Ab}$ be a covariant functor. Define $\tilde{T}:\cat{C}\to \cat{Ab}$ in the following way. For an object $A$ of $C$, let $\tilde{T}(A)$ be the free abelian group generated by the symbols $(\phi,m)$ where $\phi:M\to A$ is in $\cat{C}$, $M$ is in $\mathscr{M}$ and $m\in T(M)$. If $f:A\to B$ is in $\cat{C}$, then $\tilde{T}(f)$ is defined by $\tilde{T}(f)(\phi,m):=(f\phi,m)$. Define a natural transformation $\Phi:\tilde{T}\to T$ by $\Phi (A)(\phi,m)=T(\phi)m$.

\begin{definition}{\cite{eilenberg1953acyclic}}
\label{def:representable_functor}
The functor $T$ is said to be \emph{representable} if there is a natural transformation $\Psi:T\to \tilde{T}$ such that the composition $\Phi\Psi$ is the identity. We say $\Psi$ is a \emph{representation} of $\Phi$.
\end{definition}

\begin{theorem}{\cite[Theorem 2]{eilenberg1953acyclic}}
\label{theorem:representable_functor}
Let $K$ and $L$ be covariant functors on $\cat{C}$ with values in $\cat{Ch}(\cat{Ab})$ and let $f:K\to L$ be a map in dimension $<q$. If $K_n$ is representable for all $n\ge q$ and if $H_{n}(L(M))=0$ for all $n\ge q-1$ and for each model $M$ in $\mathscr{M}$, then $f$ admits an extension to a map $f':K\to L$ (defined in all dimensions). If $f',f'':K\to L$ are two such extension of $f$ then there is a homotopy $D:f'\simeq f''$ with $D_n=0$ for all $n<q$.
\end{theorem}

\begin{lemma}{\cite[Lemma 6.3]{eilenberg1953acyclic}}
\label{lemma:representable_functors}
Let $T,T_1:\cat{C}\to \cat{Ab}$ be functors and $\xi:T\to T_1$, $\nu:T_1\to T$ be natural transformations such that $\xi\nu$ is the identity. If $T$ is representable, then so is $T_1$.
\end{lemma}

\begin{definition}
\label{def:free_functor}
Let $T:\cat{C}\to \cat{Ch}_{\ge 0}(\cat{Ab})$ be a functor. We say $T$ is \emph{acyclic} on $\mathscr{M}$ if for each $M$ in $\mathscr{M}$, $T(M)$ is an acyclic complex in dimensions $>0$. We say $T$ is \emph{free} on $\mathscr{M}$ if there are objects $\{M_i\}_{i\in I}$ in $\mathscr{M}$ and elements $m_{i}\in T(M_{i})$ such that if $A$ is any object in $\cat{C}$ then $\{T(f)(m_i))\,|\, f:M_i\to I\}_{i\in I}$ is a basis of $T(A)$.
\end{definition}

\begin{theorem}{\cite[Theorem 9.12]{rotman2013introduction}}
\label{theorem:acyclic_model_theorem}[Acyclic models]
Let $T,T_1:\cat{C}\to \cat{Ch}_{\ge 0}(\cat{Ab})$. Suppose that $T$ is free on the models $\mathscr{M}$ and that $T_1$ is acyclic on the models $\mathscr{M}$.
\begin{enumerate}
    \item Given a natural transformation $H_0(T)\to H_0(T_1)$, there is a natural transformation $T\to T_1$ inducing it. 
    \item Given two natural transformations $T\to T_1$ inducing the same natural $H_0(T)\to H_0(T_1)$ there exists a natural chain homotopy between them.
\end{enumerate}
\end{theorem}

\begin{corollary}
\label{corollary:acyclic_model_theorem}
Let $T,T_1:\cat{C}\to \cat{Ch}_{\ge 0}(\cat{Ab})$ be functors whose images are augmented chain complexes. Suppose that $T$ and $T_1$ are both acyclic and free on the models $\mathscr{M}$. Then $T$ and $T_1$ are naturally chain homotopy equivalent. 
\end{corollary}

\printbibliography
\end{document}